\documentclass[12pt]{article}%
\usepackage{amsmath}
\usepackage{amsfonts}
\usepackage{amssymb}
\usepackage{graphicx}%
\providecommand{\U}[1]{\protect\rule{.1in}{.1in}}
\newtheorem{theorem}{Theorem}

\newtheorem{corollary}[theorem]{Corollary}

\newtheorem{definition}[theorem]{Definition}

\newtheorem{proposition}[theorem]{Proposition}
\newtheorem{remark}[theorem]{Remark}

\newenvironment{proof}[1][Proof]{\noindent\textbf{#1.} }{\ \rule{0.5em}{0.5em}}
\begin{document}

\title{Truncated Hermite polynomials}
\author{Diego Dominici \thanks{e-mail: diego.dominici@dk-compmath.jku.at}\\Johannes Kepler University Linz\\Doktoratskolleg \textquotedblleft Computational Mathematics\textquotedblright\\Altenberger Stra\ss e 69, 4040 Linz, Austria.\\Permanent address: Department of Mathematics\\State University of New York at New Paltz\\1 Hawk Dr., New Paltz, NY 12561-2443, USA
\and Francisco Marcell{\'a}n \thanks{e-mail: pacomarc@ing.uc3m.es}\\Departamento de Matem\'aticas \\Universidad Carlos III de Madrid \\Escuela Polit\'ecnica Superior \\Av. Universidad 30 \\28911 Legan\'es\\Spain}
\maketitle

\begin{abstract}
We define the family of truncated Hermite polynomials $P_{n}\left(
x;z\right)  $, orthogonal with respect to the linear functional
\[
L\left[  p\right]  =%
{\displaystyle\int\limits_{-z}^{z}}
p\left(  x\right)  e^{-x^{2}}dx,\quad p\in\mathbb{R}\left[  x\right]  ,\quad
z>0.
\]
The connection of $P_{n}\left(  x;z\right)  $ with Hermite and Rys polynomials
is stated. The semiclassical character of $P_{n}\left(  x;z\right)  $ as
polynomials of class $2$ is emphasized.

As a consequence, several properties of $P_{n}\left(  x;z\right)  $ concerning
the coefficients $\gamma_{n}\left(  z\right)  $ in the three-term recurrence
relation they satisfy as well as the moments and the Stieltjes function of $L$
are given. Ladder operators associated with such a linear functional and the
holonomic equation that the polynomials $P_{n}\left(  x;z\right)  $ satisfy
are deduced.

\end{abstract}

\section{Introduction}

The classical \emph{Hermite polynomials} (Laplace 1810, Chebyshev 1859,
Hermite 1864), are orthogonal with respect to the linear functional $L$
defined by \cite{MR0481884}%
\begin{equation}
L_{H}\left[  p\right]  =%
{\displaystyle\int\limits_{-\infty}^{\infty}}
p\left(  x\right)  e^{-x^{2}}dx,\quad p\in\mathbb{R}\left[  x\right]  .
\label{LH}%
\end{equation}
They have multiple applications in several areas of mathematics, including
Brownian motion, Gaussian quadrature, random matrices, and wavelet series.
They also appear in the framework of Sturm-Liouville equations associated with
the Schr\"{o}dinger equation of the harmonic oscillator in $\mathbb{R}$
\cite{MR2474331}$.$

Related to them are the \emph{Rys polynomials }(named after John Rys, graduate
student of Harry F. King), orthogonal with respect to the linear functional%
\begin{equation}
L_{R}\left[  p\right]  =%
{\displaystyle\int\limits_{I}}
p\left(  x\right)  e^{-ax^{2}}dx,\quad p\in\mathbb{R}\left[  x\right]  ,\quad
a>0. \label{LR}%
\end{equation}
They are usually denoted by $R_{n}\left(  x;a\right)  $ if $I=\left[
0,1\right]  $ and $J_{n}\left(  x;a\right)  $ if $I=\left[  -1,1\right]  .$
The polynomials $R_{n}\left(  x;a\right)  $ were introduced in
\cite{doi:10.1063/1.432807} to compute integrals related to electron repulsion
in molecular quantum mechanics \cite{https://doi.org/10.1002/jcc.540110809},
\cite{doi:10.1063/1.3600745}, \cite{doi:10.1063/1.461610},
\cite{doi:10.1063/1.446960},\cite{Tenno1993AnEA}. \cite{CARSKY1998266}.

The Rys polynomials have been studied by several authors, mostly from a
computational point of view, and mainly related to the implementation of
quadrature rules \cite{PMID:26613300}, \cite{doi:10.1063/1.4822929}%
,\cite{SCHWENKE2014762}. Their zeros and associated quadrature weights
(Christoffel numbers) have been extensively analyzed
\cite{doi:10.1063/1.3204437} \cite{doi:10.1021/acs.jpca.6b10004}.

A basic point is the study of the moment sequence for the linear functional
(\ref{LR}). Indeed, the moments of $L_{R}$ can be expressed in terms of the
incomplete gamma function, but for small values of $a$ they are readily
evaluated by a polynomial approximation to a non-alternating power series
expansion in $a$ that is valid over a specified range of $a.$ As a next step,
the practical evaluation of the zeros is done in terms of a low-order
approximation that is valid on finite intervals of $a,$ and by asymptotic
expansions for large $a.$ Another interesting point is the analytic
relationships between zeros and weights as well as their variation in terms of
the parameter $a$ \cite{MR474710}.

Quadrature formulas for more general linear functionals
\[%
{\displaystyle\int\limits_{-1}^{1}}
p\left(  x\right)  \left(  1-x^{2}\right)  ^{\lambda-\frac{1}{2}}e^{-zx^{2}%
}dx,\quad\lambda>-\frac{1}{2},\quad p\in\mathbb{R}\left[  x\right]  ,
\]
were considered in \cite{MR1158212}, \cite{MR661060}, \cite{MR3937492},
\cite{MR4412929}. By using a transformation of quadrature rules from the
interval $(-1,1)$ with $N$ nodes to the interval $(0,1)$ with $\frac{N+1}{2}$
nodes, the method of modified moments \cite{MR2061539} allows one to get the
coefficients in the corresponding three-term recurrence relation

In this paper, we will study the \emph{truncated Hermite polynomials}
$P_{n}\left(  x;z\right)  ,$ orthogonal with respect to the linear functional
\begin{equation}
L\left[  p\right]  =%
{\displaystyle\int\limits_{-z}^{z}}
p\left(  x\right)  e^{-x^{2}}dx,\quad p\in\mathbb{R}\left[  x\right]  ,\quad
z>0, \label{L}%
\end{equation}
and satisfying a three-term recurrence relation
\begin{equation}
xP_{n}\left(  x;z\right)  =P_{n+1}\left(  x;z\right)  +\gamma_{n}%
(z)P_{n-1}\left(  x;z\right)  ,\quad n\geq0, \label{3-term L}%
\end{equation}
with $P_{-1}=0.$

Since the Rys polynomials $J_{n}\left(  x;z^{2}\right)  $ satisfy%
\[
\frac{1}{z}%
{\displaystyle\int\limits_{-z}^{z}}
J_{n}\left(  \frac{x}{z};z^{2}\right)  J_{m}\left(  \frac{x}{z};z^{2}\right)
e^{-x^{2}}dx=\left\Vert J_{n}\right\Vert ^{2}\delta_{n,m},
\]
we can see that the polynomials $P_{n}\left(  x;z\right)  $ are related to the
Rys polynomials by%
\[
P_{n}\left(  x;z\right)  =z^{n}J_{n}\left(  \frac{x}{z};z^{2}\right)  ,\quad
h_{n}\left(  z\right)  =\left\Vert P_{n}\right\Vert ^{2}=z^{2n+1}\left\Vert
J_{n}\right\Vert ^{2}.
\]
Using the relation $\gamma_{n}(z)=\frac{h_{n}}{h_{n-1}},$ it follows that the
coefficients in the three-term recurrence relations are connected by%
\[
\gamma_{n}\left(  z\right)  =z^{2}\frac{\left\Vert J_{n}\right\Vert ^{2}%
}{\left\Vert J_{n-1}\right\Vert ^{2}}=z^{2}\gamma_{n}^{\left(  J\right)
}\left(  z\right)  .
\]

Nevertheless, our choice of (\ref{L}) is based on the fact that when
$z\rightarrow\infty$, we recover the standard Hermite linear functional
(\ref{LH}) in a direct way. Our approach has an analytic (rather than
numerical) flavor based on the $D-$semiclassical character of the linear
functional $L$. Thus, we can analyze the structure relation (ladder operator)
and the second order linear differential (holonomic) equation associated with
the corresponding sequence of orthogonal polynomials, and this ODE provides an
essential tool for an electrostatic interpretation of their zeros.\newline

The structure of the manuscript is as follows. In Section 2, we present a
basic background concerning linear functionals and orthogonal polynomials,
with a special emphasis on the symmetric and $D-$semiclassical cases. In
Section 3, we study some properties (especially the Pearson equation) that the
linear functional (\ref{L}) satisfies. The behavior of moments and the
associated Stieltjes function follows in a natural way. Section 4 is focused
on the nonlinear Laguerre-Freud equation that the coefficients of the
three-term recurrence relation $\gamma_{n}\left(  z\right)  $ satisfy. In
Section 5, the ladder operators associated with such a linear functional yield
a second order linear differential (holonomic) equation for the polynomials
$P_{n}\left(  x;z\right)  $. An electrostatic interpretation of their zeros in
terms of an external potential in a such a way they are in an equilibrium
state is given in Section 6. Finally, in Section 7, a Toda interpretation of
parameters of the three-term recurrence relation as well as of the orthogonal
polynomials in terms of the parameter $z$ is discussed. We also find a
nonlinear ODE (perhaps related to the Painlev\'{e} equations) satisfied by
$\gamma_{n}\left(  z\right)  $

\section{Basic background}

Let $\mathfrak{L}:\mathbb{R}\left[  x\right]  \rightarrow\mathbb{R}$ be a
linear functional and let $\mu_{n}$ denote the \emph{moments} of $L$ on the
monomial basis%
\begin{equation}
\mathfrak{L}\left[  x^{n}\right]  =\mu_{n}. \label{moment}%
\end{equation}
A sequence $\left\{  p_{n}\right\}  _{n\geq0},$ $\deg\left(  p_{n}\right)
=n,$ is called an \emph{orthogonal polynomial sequence} with respect to
$\mathfrak{L}$ if%
\begin{equation}
\mathfrak{L}\left[  p_{k}p_{n}\right]  =h_{n}\delta_{k,n},\quad k,n\in
\mathbb{N}_{0},\quad h_{n}\neq0, \label{ortho}%
\end{equation}
where $\delta_{k,n}$ denotes the Kronecker delta. If $h_{n}=1,$ then $\left\{
p_{n}\right\}  _{n\geq0}$ is said to be an{ }\emph{orthonormal polynomial
sequence}. Notice that it is unique with the convention that the leading
coefficient is a positive real number. In such a case, the linear functional
is said to be \emph{quasi-definite}, If $h_{n}>0$ for every $n\geq0,$ then the
linear functional is said to be \emph{positive definite}.

Let's denote by $\left\{  P_{n}\right\}  _{n\geq0}$ the sequence of
\textbf{monic} {polynomials}, {orthogonal }with respect to $\mathfrak{L}$.
From (\ref{ortho}), we see that%
\[
\mathfrak{L}\left[  xP_{k}P_{n}\right]  =0,\quad k\neq n,n\pm1,
\]
and therefore the polynomials $P_{n}\left(  x\right)  $ satisfy the
\emph{three-term recurrence relation}
\begin{equation}
xP_{n}(x)=P_{n+1}(x)+\beta_{n}P_{n}(x)+\gamma_{n}P_{n-1}(x),n\geq1,
\label{3-term}%
\end{equation}
with initial values $P_{0}(x)=1,$ \ $P_{1}(x)=x-\beta_{0}.$ The coefficients
$\beta_{n},\gamma_{n}$ are given by \cite{MR0481884}
\begin{equation}
\beta_{n}=\frac{\mathfrak{L}\left[  xP_{n}^{2}\right]  }{h_{n}},\quad
\gamma_{n}=\frac{\mathfrak{L}\left[  xP_{n}P_{n-1}\right]  }{h_{n-1}},\quad
n\geq1, \label{beta, gamma}%
\end{equation}
with initial values
\begin{equation}
\beta_{0}=\frac{\mu_{1}}{\mu_{0}}. \label{initial}%
\end{equation}
Note that using (\ref{ortho}) we have%
\[
h_{n}=\mathfrak{L}\left[  x^{n}P_{n}\right]  =\mathfrak{L}\left[
xP_{n}P_{n-1}\right]  =\gamma_{n}h_{n-1},n\geq1,
\]
and hence%
\begin{equation}
\gamma_{n}=\frac{h_{n}}{h_{n-1}},\quad n\geq1. \label{gamma-h}%
\end{equation}

\begin{definition}
A linear functional $\mathfrak{L}$ is called \emph{symmetric} if $\mu
_{2n-1}=0,$ for all $n\in\mathbb{N}.$
\end{definition}

Symmetric functionals can be characterized as follows.

\begin{theorem}
Let $\left\{  P_{n}\right\}  _{n\geq0}$ be the sequence of monic {polynomials}
{orthogonal }with respect to $\mathfrak{L}$. Then, the following statements
are equivalent:

\begin{enumerate}
\item $\mathfrak{L}$ is symmetric.

\item $\beta_{n}=0,\quad n\in\mathbb{N}.$

\item For all $n\in\mathbb{N}$
\begin{equation}
P_{n}\left(  -x\right)  =\left(  -1\right)  ^{n}P_{n}\left(  x\right)  .
\label{Pn odd}%
\end{equation}

\end{enumerate}
\end{theorem}

\begin{proof}
See \cite[Theorem 4.3]{MR0481884}.
\end{proof}

Note that if $\mathfrak{L}$ is symmetric there exist two sequences of
polynomials $\{P_{n}^{e}\left(  x\right)  \}_{n\geq0}$ and $\{P_{n}^{o}\left(
x\right)  \}_{n\geq0}$ such that
\[
P_{2n}(x)=P_{n}^{e}(x^{2}),\quad P_{2n+1}(x)=xP_{n}^{o}(x^{2}).
\]
The polynomials $P_{n}^{e}\left(  x\right)  ,P_{n}^{o}\left(  x\right)  $ are
orthogonal with respect to the linear functionals $\mathfrak{U},$
$\mathfrak{V}$ satisfying $\mathfrak{U}[x^{n}]=\mathfrak{L}[x^{2n}]$ and
$\mathfrak{V}[x^{n}]=\mathfrak{U}[x^{n+1}],$ respectively (see \cite[Theorem
8.1]{MR0481884}) .

It is clear from (\ref{L}) that $\mathfrak{L}$ is symmetric and therefore the
polynomials $P_{n}\left(  x;z\right)  $ satisfy the recurrence relation
(\ref{3-term L}) with initial conditions $P_{0}(x;z)=1,P_{1}(x;z)=x.$ If we
write
\begin{equation}
P_{n}(x;z)=x^{n}-c_{n}\left(  z\right)  x^{n-2}+d_{n}\left(  z\right)
x^{n-4}+O\left(  x^{n-6}\right)  , \label{Pn coeff}%
\end{equation}
with $c_{0}=c_{1}=0,$ $d_{0}=d_{1}=d_{2}=d_{3}=0,$ then (\ref{3-term L}) gives%
\begin{gather*}
x^{n+1}-c_{n}x^{n-1}+d_{n}x^{n-3}+O\left(  x^{n-5}\right) \\
=x^{n+1}-c_{n+1}x^{n-1}+d_{n+1}x^{n-3}+O\left(  x^{n-5}\right) \\
+\gamma_{n}\left(  x^{n-1}-c_{n-1}x^{n-3}+d_{n+1}x^{n-5}+O\left(
x^{n-7}\right)  \right)  ,
\end{gather*}
and comparing powers of $x$ we get%
\begin{equation}
c_{n}=c_{n+1}-\gamma_{n},\quad d_{n}=d_{n+1}-\gamma_{n}c_{n-1}. \label{cn dn}%
\end{equation}
We conclude that
\[
c_{n}\left(  z\right)  =%
{\displaystyle\sum\limits_{k=1}^{n-1}}
\gamma_{k}\left(  z\right)  ,\quad n\geq2,
\]
and%
\[
d_{n}\left(  z\right)  =%
{\displaystyle\sum\limits_{k=3}^{n-1}}
\gamma_{k}\left(  z\right)  c_{k-1}\left(  z\right)  ,\quad n\geq4.
\]
Reversing (\ref{Pn coeff}), we obtain%
\begin{equation}%
\begin{tabular}
[c]{l}%
$x^{n}=P_{n}(x;z)+c_{n}\left(  z\right)  P_{n-2}(x;z)$\\
$-\left[  d_{n}\left(  z\right)  -c_{n}\left(  z\right)  c_{n-2}\left(
z\right)  \right]  P_{n-4}(x;z)+O\left(  x^{n-6}\right)  .$%
\end{tabular}
\label{x^n}%
\end{equation}

Taking the derivative with respect to $x$ in (\ref{Pn coeff}), we have%
\[
\partial_{x}P_{n}=nx^{n-1}-\left(  n-2\right)  c_{n}x^{n-3}+\left(
n-4\right)  d_{n}x^{n-5}+O\left(  x^{n-7}\right)  ,
\]
and using (\ref{x^n}) we see that%
\[
\partial_{x}P_{n}=n\left(  P_{n-1}+c_{n-1}P_{n-3}\right)  -\left(  n-2\right)
c_{n}P_{n-3}+O\left(  x^{n-5}\right)  .
\]
Since $c_{n-1}-c_{n}=-\gamma_{n-1},$ we conclude that%
\begin{equation}
\partial_{x}P_{n}(x;z)=nP_{n-1}(x;z)+\left[  2c_{n}\left(  z\right)
-n\gamma_{n-1}\right]  P_{n-3}(x;z)+O\left(  x^{n-5}\right)  . \label{derP}%
\end{equation}

An interesting family of linear functionals is the so called $D-$%
\emph{semiclassical} (with respect to the derivative operator). In such a
case, $\mathfrak{L}$ satisfies a first order linear differential equation
(\emph{Pearson equation})%
\[
\partial_{x}^{\ast}\left(  \phi\mathfrak{L}\right)  +\psi\mathfrak{L}=0,
\]
where $\phi\left(  x\right)  $ is a monic polynomial, $\psi\left(  x\right)  $
is a polynomial of degree at least $1$, and the \emph{adjoints} of the
derivative and multiplication operators are defined by \cite{MR4331433}%
\begin{equation}
\left(  \partial_{x}^{\ast}\mathfrak{L}\right)  \left[  p\right]
=-\mathfrak{L}\left[  \partial_{x}p\right]  ,\quad\left(  x\mathfrak{L}%
\right)  \left[  p\right]  =\mathfrak{L}\left[  xp\right]  . \label{adjoints}%
\end{equation}
Notice that a semiclassical linear functional satisfies many Pearson equations
taking into account the choices of the polynomials $\phi,\psi.$ The
\textbf{minimal degree} choice of $\phi,\psi$ yields the definition of the
\emph{class} of a semiclassical linear functional $\mathfrak{L}$ as%
\[
s=\max\left\{  \deg(\phi)-2,\deg(\psi)-1\right\}  .
\]
Notice that $D-$\emph{classical} linear functionals (Hermite, Laguerre,
Jacobi, Bessel) are semiclassical of class $s=0$. The description of
$D-$semiclassical linear functional of class $s=1$ has been done in
\cite{MR1186737}. Characterizations of the $D-$semiclassical orthogonal
polynomial sequences were pioneered by Pascal Maroni \cite{MR932783},
\cite{MR1270222}, \cite{MR1246855}, \cite{MR1711161}.

\section{Truncated Hermite linear functional}

In this section we study the truncated Hermite linear functional $L$ defined
by (\ref{L}), which is a $D-$semiclassical functional of class $s=2.$ In order
to prove it, we will first find the corresponding Pearson equation. Next, we
will deduce a second order linear recurrence equation satisfied by its
moments. By using the $z$-transform of the sequence of moments (Stieltjes
function), we will get a first order linear differential equation that the
Stieltjes function satisfies.

\subsection{Pearson equation}

\begin{proposition}
Let $p\in\mathbb{R}\left[  x\right]  $ and $\phi\left(  x;z\right)  ,$
$\psi\left(  x;z\right)  $ be defined by
\begin{equation}
\phi\left(  x;z\right)  =x^{2}-z^{2},\quad\psi\left(  x;z\right)
=2x\phi\left(  x;z\right)  . \label{phi}%
\end{equation}
The functional $L$ defined by (\ref{L}) satisfies the Pearson equation%
\begin{equation}
L\left[  \partial_{x}\left(  \phi p\right)  \right]  =L\left[  \psi p\right]
, \label{Pearson1}%
\end{equation}
or equivalently%
\begin{equation}
L\left[  \phi\partial_{x}p\right]  =L\left[  2x\left(  \phi-1\right)
p\right]  . \label{Pearson2}%
\end{equation}

\end{proposition}

\begin{proof}
Let $p\in\mathbb{R}\left[  x\right]  .$ We have%
\begin{gather*}
L\left[  \partial_{x}\left(  \phi p\right)  \right]  =%
{\displaystyle\int\limits_{-z}^{z}}
\partial_{x}\left(  \phi p\right)  e^{-x^{2}}dx\\
=\left[  \phi\left(  x\right)  p\left(  x\right)  e^{-x^{2}}\right]  _{-z}%
^{z}-%
{\displaystyle\int\limits_{-z}^{z}}
-2x\phi\left(  x\right)  p\left(  x\right)  e^{-x^{2}}dx=L\left[  2x\phi
p\right]  ,
\end{gather*}
and we obtain (\ref{Pearson1}). Using the product rule, we see that equation
(\ref{Pearson2}) follows immediately from (\ref{Pearson1}).
\end{proof}

Using the adjoint operators defined by (\ref{adjoints}), we can write
$\phi\partial_{x}^{\ast}L\left[  p\right]  =-L\left[  \partial_{x}\left(  \phi
p\right)  \right]  ,$ and therefore the Pearson equation (\ref{Pearson1}) has
the form
\begin{equation}
\left(  \phi\partial_{x}^{\ast}+\psi\right)  L=0. \label{PL1}%
\end{equation}

\begin{remark}
For the linear functional (\ref{LH}) associated with the Hermite polynomials,
we have the Pearson equation
\begin{equation}
\partial_{x}^{\ast}L_{H}+2xL_{H}=0. \label{Pearson H}%
\end{equation}
Multiplying (\ref{Pearson H}) by $\phi\left(  x,z\right)  ,$ we get%
\begin{equation}
0=\phi\partial_{x}^{\ast}L_{H}+2x\phi L_{H}=\left(  \phi\partial_{x}^{\ast
}+\psi\right)  L_{H}, \label{PH1}%
\end{equation}
which yields another Pearson equation satisfied by $L_{H}$ but that has no
minimal degree. Note that (\ref{PH1}) is equivalent to (\ref{PL1}).
\end{remark}

\subsection{Moments}

It follows from the definition of $L$ that the odd moments are zero. Setting
$s=x^{2}$ in (\ref{moment}), we have%

\begin{equation}
\mu_{2n}\left(  z\right)  =2%
{\displaystyle\int\limits_{0}^{z}}
x^{2n}e^{-x^{2}}dx=%
{\displaystyle\int\limits_{0}^{z^{2}}}
s^{n-\frac{1}{2}}e^{-s}ds=\widehat{\gamma}\left(  n+\frac{1}{2},z^{2}\right)
, \label{mu 2n}%
\end{equation}
where the \emph{incomplete gamma function} $\widehat{\gamma}\left(
a,z\right)  $ is defined by \cite[8.2.1]{MR2723248}
\[
\widehat{\gamma}\left(  a,z\right)  =%
{\displaystyle\int\limits_{0}^{z}}
t^{a-1}e^{-t}dt.
\]
Note that we use the nonstandard notation $\widehat{\gamma}\left(  a,z\right)
$ to distinguish the incomplete gamma function from the coefficient
$\gamma_{n}\left(  z\right)  $ in the recurrence relation (\ref{3-term L}).

The function $\widehat{\gamma}\left(  a,z\right)  $ has the hypergeometric
representation \cite[8.5.1]{MR2723248}%
\begin{equation}
\widehat{\gamma}\left(  a,z\right)  =a^{-1}z^{a}e^{-z}\ _{1}F_{1}\left(
\begin{array}
[c]{c}%
1\\
a+1
\end{array}
;z\right)  , \label{hyper}%
\end{equation}
where the (generalized) \emph{hypergeometric function}$\ _{p}F_{q}$\ is
defined by \cite[16.2.1]{MR2723248}%
\[
_{p}F_{q}\left(
\begin{array}
[c]{c}%
a_{1},\ldots,a_{p}\\
b_{1},\ldots,b_{q}%
\end{array}
;z\right)  =%
{\displaystyle\sum\limits_{x=0}^{\infty}}
\frac{\left(  a_{1}\right)  _{x}\cdots\left(  a_{p}\right)  _{x}}{\left(
b_{1}\right)  _{x}\cdots\left(  b_{q}\right)  _{x}}\frac{z^{x}}{x!},
\]
and the \emph{Pochhammer symbol} is defined by \cite[5.2.4]{MR2723248}
\begin{equation}
\left(  c\right)  _{n}=%
{\displaystyle\prod\limits_{j=0}^{n-1}}
\left(  c+j\right)  ,\quad n\in\mathbb{N},\quad\left(  c\right)  _{0}=1.
\label{poch poly}%
\end{equation}

Using (\ref{hyper}) in (\ref{mu 2n}), we get%
\begin{equation}
\mu_{2n}\left(  z\right)  =\frac{2}{2n+1}z^{2n+1}e^{-z^{2}}\ _{1}F_{1}\left(
\begin{array}
[c]{c}%
1\\
n+\frac{3}{2}%
\end{array}
;z^{2}\right)  , \label{mu hyper}%
\end{equation}
and for $n=0$ we have%
\begin{equation}
\mu_{0}\left(  z\right)  =2%
{\displaystyle\int\limits_{0}^{z}}
e^{-x^{2}}dx=\sqrt{\pi}\operatorname{erf}\left(  z\right)  , \label{mu0}%
\end{equation}
where the $\operatorname{erf}\left(  z\right)  $ denotes the \emph{error
function} \cite[7.2.1]{MR2723248}. Using the recurrence relation
\cite[8.8.1]{MR2723248}%
\[
\widehat{\gamma}\left(  a+1,z\right)  =a\widehat{\gamma}\left(  a,z\right)
-z^{a}e^{-z},
\]
we get%
\begin{equation}
\mu_{2n+2}\left(  z\right)  =\left(  n+\frac{1}{2}\right)  \mu_{2n}\left(
z\right)  -z^{2n+1}e^{-z^{2}}. \label{mu req}%
\end{equation}
In particular,%
\[
\mu_{2}\left(  z\right)  =\frac{1}{2}\mu_{0}\left(  z\right)  -\sqrt
{z}e^{-z^{2}}=\frac{\sqrt{\pi}}{2}\operatorname{erf}\left(  z\right)
-ze^{-z^{2}}.
\]

To obtain a second order homogeneous recurrence equation that the moments
satisfy, we can use the Pearson equation (\ref{Pearson2}).

\begin{proposition}
Let $u_{n}\left(  z\right)  $ be defined by
\begin{equation}
u_{n}\left(  z\right)  =\mu_{2n}\left(  z\right)  . \label{un}%
\end{equation}
Then, $u_{n}\left(  z\right)  $ satisfies the recurrence equation%
\begin{equation}
2u_{n+2}-\left(  2n+3+2z^{2}\right)  u_{n+1}+\left(  2n+1\right)  z^{2}%
u_{n}=0, \label{un req}%
\end{equation}
with initial conditions%
\[
u_{0}=\sqrt{\pi}\operatorname{erf}\left(  z\right)  ,\quad u_{1}=\frac
{\sqrt{\pi}}{2}\operatorname{erf}\left(  z\right)  -ze^{-z^{2}}.
\]

\end{proposition}

\begin{proof}
Using (\ref{Pearson2}) with $p\left(  x\right)  =x^{2n+1}$, we have%
\[
L\left[  \left(  2n+1\right)  \left(  x^{2}-z^{2}\right)  x^{2n}\right]
=L\left[  2x\left(  x^{2}-z^{2}-1\right)  x^{2n+1}\right]  ,
\]
or%
\[
\left(  2n+1\right)  \left(  \mu_{2n+2}-z^{2}\mu_{2n}\right)  =2\left[
\mu_{2n+4}-\left(  z^{2}+1\right)  \mu_{2n+2}\right]  .
\]

\end{proof}

\begin{remark}
From the asymptotic expansion \cite[8.11.2]{MR2723248}%
\[
\widehat{\gamma}\left(  a,z\right)  \sim\Gamma\left(  a\right)  \left[
1-z^{a-1}e^{-z}%
{\displaystyle\sum\limits_{k\geq0}}
\frac{z^{-k}}{\Gamma\left(  a-k\right)  }\right]  ,\quad z\rightarrow\infty,
\]
in (\ref{mu 2n}), we see that for fixed $n$%
\begin{equation}
u_{n}\left(  z\right)  \sim\Gamma\left(  n+\frac{1}{2}\right)  \left[
1-e^{-z^{2}}%
{\displaystyle\sum\limits_{k\geq0}}
\frac{z^{2\left(  n-k\right)  -1}}{\Gamma\left(  n+\frac{1}{2}-k\right)
}\right]  ,\quad z\rightarrow\infty. \label{un asymp}%
\end{equation}

The even moments of the Hermite polynomials are \cite[5.9.1]{MR2723248}%
\[
\mu_{2n}^{H}=%
{\displaystyle\int\limits_{0}^{\infty}}
s^{n-\frac{1}{2}}e^{-s}ds=\Gamma\left(  n+\frac{1}{2}\right)  ,
\]
and therefore we can rewrite (\ref{un asymp}) as the ratio asymptotics
\[
\frac{\mu_{2n}\left(  z\right)  }{\mu_{2n}^{H}}\sim1-e^{-z^{2}}%
{\displaystyle\sum\limits_{k\geq0}}
\frac{z^{2\left(  n-k\right)  -1}}{\Gamma\left(  n+\frac{1}{2}-k\right)
},\quad z\rightarrow\infty.
\]

\end{remark}

\subsection{Stieltjes function}

The \emph{Stieltjes function} associated with a linear functional
$\mathfrak{L}$ is defined by \cite{MR4331433}, \cite{MR4331433},
\cite{MR1270222}%
\begin{equation}
S\left(  t\right)  =\mathfrak{L}\left[  \frac{1}{t-x}\right]  =%
{\displaystyle\sum\limits_{n\geq0}}
\frac{\mu_{n}}{t^{n+1}}, \label{S}%
\end{equation}
where the sum is a formal power series.

\begin{proposition}
The Stieltjes function $S\left(  t;z\right)  $ associated with the linear
functional $L$ satisfies the first order \textbf{nonhomogeneous} ODE%
\begin{equation}
\phi\left(  t;z\right)  \partial_{t}S+\psi\left(  t;z\right)  S=\left[
2\phi\left(  t;z\right)  -1\right]  u_{0}\left(  z\right)  +2u_{1}\left(
z\right)  , \label{ODES}%
\end{equation}
where $\phi\left(  x;z\right)  $ and $\psi\left(  x;z\right)  $ were defined
in (\ref{phi}).
\end{proposition}

\begin{proof}
For the linear functional $L,$ the Stieltjes function reads%
\begin{equation}
S\left(  t;z\right)  =%
{\displaystyle\sum\limits_{n\geq0}}
\frac{u_{n}\left(  z\right)  }{t^{2n+1}}, \label{S1}%
\end{equation}
and therefore%
\begin{equation}%
{\displaystyle\sum\limits_{n\geq0}}
\left(  2n+1\right)  \frac{u_{n}(z)}{t^{2n+1}}=-t\partial_{t}S, \label{dS}%
\end{equation}
i. e.,
\begin{equation}%
{\displaystyle\sum\limits_{n\geq0}}
\frac{u_{n+k}(z)}{t^{2n+1}}=%
{\displaystyle\sum\limits_{n\geq k}}
\frac{u_{n}(z)}{t^{2n-2k+1}}=t^{2k}\left(  S-%
{\displaystyle\sum\limits_{n=0}^{k-1}}
\frac{u_{n}(z)}{t^{2n+1}}\right)  . \label{u n+k}%
\end{equation}
Using (\ref{dS}) and (\ref{u n+k}), we have%
\begin{gather*}
2t^{4}\left[  S-\left(  \frac{u_{0}(z)}{t}+\frac{u_{1}(z)}{t^{3}}\right)
\right]  +t^{2}\left(  t\partial_{t}S+\frac{u_{0}(z)}{t}\right) \\
-2z^{2}t^{2}\left(  S-\frac{u_{0}(z)}{t}\right)  -z^{2}t\partial_{t}S=0,
\end{gather*}
since
\[%
{\displaystyle\sum\limits_{n\geq0}}
\left(  2n+3\right)  \frac{u_{n+1}(z)}{t^{2n+1}}=t^{2}%
{\displaystyle\sum\limits_{n\geq1}}
\left(  2n+1\right)  \frac{u_{n}(z)}{t^{2n+1}}=t^{2}\left(  -t\partial
_{t}S-\frac{u_{0}(z)}{t}\right)  .
\]
After simplification, we obtain%
\[
\left(  t^{2}-z^{2}\right)  \left(  \partial_{t}S+2tS\right)  =\left(
2t^{2}-2z^{2}-1\right)  u_{0}(z)+2u_{1}(z).
\]

\end{proof}

Note that differentiating (\ref{ODES}) with respect to $t,$ we get
\[
\partial_{t}\left[  \left(  t^{2}-z^{2}\right)  \left(  \partial
_{t}S+2tS\right)  \right]  =4tu_{0}(z),
\]
and, therefore,%
\[
\partial_{t}\left(  \frac{\partial_{t}\left[  \left(  t^{2}-z^{2}\right)
\left(  \partial_{t}S+2tS\right)  \right]  }{t}\right)  =0.
\]
Thus, the function $S\left(  t;z\right)  $ satisfies the third order
\textbf{homogeneous} linear ODE with polynomial coefficients
\begin{gather*}
t\left(  t^{2}-z^{2}\right)  \partial_{t}^{3}S+\left(  2t^{4}-2t^{2}%
z^{2}+3t^{2}+z^{2}\right)  \partial_{t}^{2}S\\
+2t(5t^{2}-z^{2})\partial_{t}S+2(3t^{2}+z^{2})S=0.
\end{gather*}

\begin{remark}
For the Hermite polynomials, we have \cite[7.7.2]{MR2723248} \qquad%
\[
S_{H}\left(  t\right)  =%
{\displaystyle\int\limits_{-\infty}^{\infty}}
\frac{e^{-x^{2}}}{t-x}dx=-\mathrm{i}\pi\omega\left(  t\right)  ,\quad
\operatorname{Im}\left(  t\right)  >0\mathbf{,}%
\]
where $\mathrm{i}^{2}=-1,$ and the function $\omega\left(  t\right)  $ is
defined by \cite[7.2.3]{MR2723248}%
\[
\omega\left(  t\right)  =e^{-t^{2}}\left[  1-\operatorname{erf}\left(
-\mathrm{i}t\right)  \right]  .
\]
The function $\omega\left(  t\right)  $ satisfies \cite[7.10.2]{MR2723248}%
\[
\omega^{\prime}\left(  t\right)  =-2t\omega\left(  t\right)  +\frac
{2\mathrm{i}}{\sqrt{\pi}},
\]
and therefore%
\begin{equation}
S_{H}^{\prime}\left(  t\right)  =-2tS_{H}\left(  t\right)  +2\sqrt{\pi
}=-2tS_{H}\left(  t\right)  +2\mu_{0}^{H}. \label{ODESH}%
\end{equation}
Comparing (\ref{ODESH}) with (\ref{ODES}), we see that the Stieltjes functions
of the functionals $L$ and $L_{H}$ are related by%
\[
0=\phi\left(  t;z\right)  \left(  \partial_{t}S_{H}+2tS_{H}-2\mu_{0}%
^{H}\right)  =\phi\left(  t;z\right)  \left(  \partial_{t}S+2tS-2u_{0}\right)
+u_{0}-2u_{1}.
\]

\end{remark}

\section{Laguerre-Freud equations}

The (generally nonlinear) equations satisfied by the coefficients of the
three-term recurrence relation (\ref{3-term}) are known in the literature as
\emph{Laguerre-Freud equations} (see \cite{MR1272122}). They can be consider
discrete analogues of the Painlev\'{e} equations \cite{MR1705232}.\newline

First of all, we will find a second order nonlinear difference equation that
the parameters of the three term recurrence relation satisfy.\newline

\begin{theorem}
The coefficients $\gamma_{n}\left(  z\right)  $ satisfy the Laguerre-Freud
equation
\begin{equation}
\frac{z^{2}}{2}=\gamma_{n}\left(  \gamma_{n-1}+\gamma_{n}-z^{2}+\frac{1}%
{2}-n\right)  -\gamma_{n+1}\left(  \gamma_{n+1}+\gamma_{n+2}-z^{2}-n-\frac
{3}{2}\right)  . \label{LaguerreFreud1}%
\end{equation}

\end{theorem}

\begin{proof}
Taking $p(x)=P_{n}P_{n+1}$ in (\ref{Pearson1}), we get
\begin{equation}
L\left[  \partial_{x}\left(  \phi P_{n}P_{n+1}\right)  \right]  =L\left[  \psi
P_{n}P_{n+1}\right]  . \label{L1}%
\end{equation}
The left hand side gives%
\[
L\left[  \partial_{x}\left(  \phi P_{n}P_{n+1}\right)  \right]  =2L\left[
xP_{n}P_{n+1}\right]  +L\left[  (x^{2}-z^{2})\left(  P_{n+1}\partial_{x}%
P_{n}+P_{n}\partial_{x}P_{n+1}\right)  \right]
\]
and using (\ref{ortho}), we have%
\begin{align*}
2L\left[  xP_{n}P_{n+1}\right]   &  =2h_{n+1},\quad L\left[  \left(
x^{2}-z^{2}\right)  P_{n+1}\partial_{x}P_{n}\right]  =nh_{n+1},\\
L\left[  -z^{2}P_{n}\partial_{x}P_{n+1}\right]   &  =-z^{2}\left(  n+1\right)
h_{n}.
\end{align*}
From (\ref{3-term L}), we see that
\begin{equation}
x^{2}P_{n}=P_{n+2}+\left(  \gamma_{n+1}+\gamma_{n}\right)  P_{n}+\gamma
_{n-1}\gamma_{n}P_{n-2}, \label{x^2}%
\end{equation}
and using (\ref{derP}) we get%
\[
\partial_{x}P_{n+1}=\left(  n+1\right)  P_{n}+\left[  2c_{n+1}-\left(
n+1\right)  \gamma_{n}\right]  P_{n-2}+O\left(  x^{n-4}\right)  .
\]
Hence,
\[
L\left[  x^{2}P_{n}\partial_{x}P_{n+1}\right]  =\left(  n+1\right)  \left(
\gamma_{n+1}+\gamma_{n}\right)  h_{n}+\left[  2c_{n+1}-\left(  n+1\right)
\gamma_{n}\right]  \gamma_{n-1}\gamma_{n}h_{n-2}.
\]
Using (\ref{gamma-h}), we conclude that%
\[
L\left[  x^{2}P_{n}\partial_{x}P_{n+1}\right]  =\left(  n+1\right)  \left(
\gamma_{n+1}+\gamma_{n}\right)  h_{n}+\left[  2c_{n+1}-\left(  n+1\right)
\gamma_{n}\right]  h_{n},
\]
and therefore%
\[
L\left[  \partial_{x}\left(  \phi P_{n}P_{n+1}\right)  \right]  =\left[
\left(  2n+3\right)  \gamma_{n+1}+2c_{n+1}-\left(  n+1\right)  z^{2}\right]
h_{n}.
\]

The right hand side in (\ref{L1}) reads%
\[
L\left[  \psi P_{n}P_{n+1}\right]  =L\left[  2x(x^{2}-z^{2})P_{n}%
P_{n+1}\right]  ,
\]
and since (\ref{3-term L}) gives%
\begin{equation}%
\begin{tabular}
[c]{l}%
$x^{3}P_{n}=P_{n+3}+\left(  \gamma_{n+2}+\gamma_{n+1}+\gamma_{n}\right)
P_{n+1}$\\
$+\gamma_{n}\left(  \gamma_{n+1}+\gamma_{n}+\gamma_{n-1}\right)
P_{n-1}+\gamma_{n}\gamma_{n-1}\gamma_{n-2}P_{n-3},$%
\end{tabular}
\label{x^3}%
\end{equation}
we have%
\[
L\left[  \psi P_{n}P_{n+1}\right]  =2\left(  \gamma_{n+2}+\gamma_{n+1}%
+\gamma_{n}\right)  h_{n+1}-2z^{2}h_{n+1}%
\]
or, using (\ref{gamma-h}),%
\[
L\left[  \psi P_{n}P_{n+1}\right]  =2\gamma_{n+1}\left(  \gamma_{n+2}%
+\gamma_{n+1}+\gamma_{n}-z^{2}\right)  h_{n}.
\]

Thus,
\begin{equation}
\left(  2n+3\right)  \gamma_{n+1}+2c_{n+1}-\left(  n+1\right)  z^{2}%
=2\gamma_{n+1}\left(  \gamma_{n+2}+\gamma_{n+1}+\gamma_{n}-z^{2}\right)  ,
\label{Eq1}%
\end{equation}
and shifting $n\rightarrow n-1,$%
\begin{equation}
\left(  2n+1\right)  \gamma_{n}+2c_{n}-nz^{2}=2\gamma_{n}\left(  \gamma
_{n+1}+\gamma_{n}+\gamma_{n-1}-z^{2}\right)  . \label{Eq2}%
\end{equation}
Subtracting (\ref{Eq2}) from (\ref{Eq1}) and using (\ref{cn dn}), we obtain%
\[
\left(  2n+3\right)  \gamma_{n+1}-\left(  2n-1\right)  \gamma_{n}%
-z^{2}=2\gamma_{n+1}\left(  \gamma_{n+2}+\gamma_{n+1}-z^{2}\right)
-2\gamma_{n}\left(  \gamma_{n}+\gamma_{n-1}-z^{2}\right)
\]
and the result follows.
\end{proof}

\begin{remark}
Notice that the nonlinear equation of order $2$ (\ref{LaguerreFreud1})
involves $4$ consecutive terms of the sequence of parameters $\gamma
_{n}\left(  z\right)  .$
\end{remark}

As an alternative to (\ref{LaguerreFreud1}) we will next deduce a third order
nonlinear difference equation involving $3$ consecutive terms of the sequence
of parameters $\gamma_{n}\left(  z\right)  .$ Thus, the computation of them is
more accurate.\newline

\begin{theorem}
The coefficients $\gamma_{n}\left(  z\right)  $ satisfy the equation
\begin{equation}
\gamma_{n}\left(  n+\frac{1}{2}-\gamma_{n}-\gamma_{n+1}\right)  \left(
n-\frac{1}{2}-\gamma_{n}-\gamma_{n-1}\right)  =z^{2}\left(  \frac{n}{2}%
-\gamma_{n}\right)  ^{2}. \label{LF}%
\end{equation}

\end{theorem}

\begin{proof}
From (\ref{beta, gamma}), we have%
\[
h_{n-1}\left(  z\right)  \gamma_{n}\left(  z\right)  =L\left[  xP_{n}%
P_{n-1}\right]  =-\frac{1}{2}\left[  P_{n}P_{n-1}e^{-x^{2}}\right]  _{-z}%
^{z}+\frac{1}{2}L\left[  \partial_{x}\left(  P_{n}P_{n-1}\right)  \right]  ,
\]
while (\ref{Pn odd}) gives
\begin{align*}
\left[  P_{n}P_{n-1}e^{-x^{2}}\right]  _{-z}^{z}  &  =P_{n}\left(  z;z\right)
P_{n-1}\left(  z;z\right)  e^{-z^{2}}-\left(  -1\right)  ^{2n-1}P_{n}\left(
z;z\right)  P_{n-1}\left(  z;z\right)  e^{-z^{2}}\\
&  =2P_{n}\left(  z;z\right)  P_{n-1}\left(  z;z\right)  e^{-z^{2}}.
\end{align*}
But using (\ref{ortho}), we see that $L\left[  P_{n}\partial_{x}%
P_{n-1}\right]  =0,$ and therefore%
\begin{equation}
h_{n-1}\left(  z\right)  \gamma_{n}\left(  z\right)  =-P_{n}\left(
z;z\right)  P_{n-1}\left(  z;z\right)  e^{-z^{2}}+\frac{1}{2}L\left[
\partial_{x}P_{n}P_{n-1}\right]  . \label{LF3}%
\end{equation}
Since $P_{n}\left(  x;z\right)  =x^{n}+O\left(  x^{n-1}\right)  ,$ we have%
\begin{equation}
\partial_{x}P_{n}\left(  x;z\right)  =nx^{n-1}+O\left(  x^{n-2}\right)
=nP_{n-1}\left(  x;z\right)  +O\left(  x^{n-2}\right)  , \label{Pn diff}%
\end{equation}
and hence%
\begin{equation}
L\left[  P_{n-1}\partial_{x}P_{n}\right]  =nh_{n-1}\left(  z\right)  .
\label{LF4}%
\end{equation}
From (\ref{LF3}) and (\ref{LF4}), we conclude that
\begin{equation}
P_{n}\left(  z;z\right)  P_{n-1}\left(  z;z\right)  e^{-z^{2}}=\left[
\frac{n}{2}-\gamma_{n}\left(  z\right)  \right]  h_{n-1}\left(  z\right)  .
\label{LF5}%
\end{equation}

On the other hand, if follows from (\ref{3-term L}) that%
\[
x^{2}P_{n}^{2}(x;z)=\left(  P_{n+1}(x;z)+\gamma_{n}P_{n-1}(x;z)\right)
^{2}=P_{n+1}^{2}(x;z)+2\gamma_{n}P_{n-1}(x;z)P_{n+1}(x;z)+\gamma_{n}%
^{2}P_{n-1}^{2}(x;z).
\]
Thus (\ref{ortho}) and (\ref{gamma-h}) give%
\begin{equation}
L\left[  x^{2}P_{n}^{2}\right]  =h_{n+1}\left(  z\right)  +\gamma_{n}%
^{2}\left(  z\right)  h_{n-1}\left(  z\right)  =h_{n+1}\left(  z\right)
+\gamma_{n}\left(  z\right)  h_{n}\left(  z\right)  . \label{x2P2}%
\end{equation}
But%
\[
L\left[  \partial_{x}\left(  xP_{n}^{2}\right)  \right]  =\left[  xP_{n}%
^{2}e^{-x^{2}}\right]  _{-z}^{z}+2L\left[  x^{2}P_{n}^{2}\right]  =2zP_{n}%
^{2}\left(  z;z\right)  e^{-z^{2}}+2L\left[  x^{2}P_{n}^{2}\right]  ,
\]
and, as a consequence, we obtain%
\begin{equation}
L\left[  \partial_{x}\left(  xP_{n}^{2}\right)  \right]  =2zP_{n}^{2}\left(
z;z\right)  e^{-z^{2}}+2\left[  h_{n+1}\left(  z\right)  +\gamma_{n}%
^{2}\left(  z\right)  h_{n-1}\left(  z\right)  \right]  . \label{LF1}%
\end{equation}

On the other hand, since%
\[
xP_{n}\left(  x;z\right)  \partial_{x}P_{n}\left(  x;z\right)  =P_{n}\left(
x;z\right)  \left[  nx^{n}+O\left(  x^{n-1}\right)  \right]  ,
\]
and using (\ref{ortho}) we get%
\begin{equation}
L\left[  \partial_{x}\left(  xP_{n}^{2}\right)  \right]  =L\left[  P_{n}%
^{2}\right]  +L\left[  2xP_{n}\partial_{x}P_{n}\right]  =(2n+1) h_{n}\left(
z\right)  . \label{LF2}%
\end{equation}
From (\ref{LF1}) and (\ref{LF2}), we conclude that%
\begin{equation}
P_{n}^{2}\left(  z;z\right)  e^{-z^{2}}=\frac{\left(  2n+1\right)
h_{n}\left(  z\right)  -2\left[  h_{n+1}\left(  z\right)  +\gamma_{n}%
^{2}\left(  z\right)  h_{n-1}\left(  z\right)  \right]  }{2z}. \label{LF6}%
\end{equation}

Squaring (\ref{LF5}), we have%
\[
P_{n}^{2}\left(  z;z\right)  P_{n-1}^{2}\left(  z;z\right)  e^{-2z^{2}%
}=\left[  \frac{n}{2}-\gamma_{n}\left(  z\right)  \right]  ^{2}h_{n-1}%
^{2}\left(  z\right)  ,
\]
and using (\ref{LF6}), we obtain%
\begin{align*}
&  \frac{\left(  2n+1\right)  h_{n}-2\left(  h_{n+1}+\gamma_{n}^{2}%
h_{n-1}\right)  }{2z}\frac{\left(  2n-1\right)  h_{n-1}-2\left(  h_{n}%
+\gamma_{n-1}^{2}h_{n-2}\right)  }{2z}\\
&  =\left(  \frac{n}{2}-\gamma_{n}\right)  ^{2}h_{n-1}^{2}.
\end{align*}
Dividing by $h_{n-1}^{2}\left(  z\right)  $ and using (\ref{gamma-h}), we get%
\[
\frac{\gamma_{n}\left(  2n-2\gamma_{n}-2\gamma_{n+1}+1\right)  \left(
2n-2\gamma_{n}-2\gamma_{n-1}-1\right)  }{4z^{2}}=\left(  \frac{n}{2}%
-\gamma_{n}\right)  ^{2},
\]
since%
\[
\frac{h_{n+1}\left(  z\right)  }{h_{n}\left(  z\right)  }=\gamma_{n+1}(z),
n\geq0.
\]

\end{proof}

\begin{corollary}
Let $g_{n}\left(  z\right)  $ be defined by%
\begin{equation}
g_{n}\left(  z\right)  =\frac{n}{2}-\gamma_{n}\left(  z\right)  . \label{gn}%
\end{equation}
Then, $g_{n}\left(  z\right)  $ satisfies the nonlinear recurrence
\begin{equation}
\left(  \frac{n}{2}-g_{n}\right)  \allowbreak\left(  g_{n}+g_{n+1}\right)
\allowbreak\left(  g_{n}+g_{n-1}\right)  =z^{2}g_{n}^{2}. \label{LF gn}%
\end{equation}

\end{corollary}

\begin{remark}
For the Hermite polynomials $H_{n}\left(  x\right)  $, we have%
\[
xH_{n}\left(  x\right)  =H_{n+1}\left(  x\right)  +\frac{n}{2}H_{n-1}\left(
x\right)  ,
\]
and therefore%
\[
\gamma_{H}\left(  n\right)  =\frac{n}{2}.
\]
Thus, from (\ref{gn}) we see that $g_{n}\rightarrow0$ as $z\rightarrow\infty.$
\end{remark}

To obtain an asymptotic expansion of $g_{n}\left(  z\right)  $ as
$n\rightarrow\infty,$ we can use the nonlinear recurrence (\ref{LF gn}).

\begin{theorem}
For $z=O\left(  1\right)  ,$ we have
\begin{equation}
g_{n}\left(  z\right)  \sim\frac{n}{2}-\frac{z^{2}}{4}-\frac{z^{2}}{16}%
n^{-2}-\frac{z^{4}}{16}n^{-3}-\frac{z^{2}}{64}(1+3z^{4})n^{-4}+ O (n^{-5})
\label{gn asymp}%
\end{equation}
as $n\rightarrow\infty.$
\end{theorem}

\begin{proof}
Replacing
\[
g_{n}\left(  z\right)  =%
{\displaystyle\sum\limits_{k\geq-1}}
\frac{\xi_{k}\left(  z\right)  }{n^{k}}%
\]
in (\ref{LF gn}) and comparing coefficients of $n$, we get%
\[
2\xi_{-1}^{2}\left(  1-2\xi_{-1}\right)  =0,\quad O\left(  n^{3}\right)  ,
\]
and therefore $\xi_{-1}=0,$ or $\xi_{-1}\left(  z\right)  =\frac{1}{2}.$ The
solution $\xi_{-1}=0$ leads to $\xi_{k}\left(  z\right)  =0$ for all $k,$ and
hence we choose $\xi_{-1}\left(  z\right)  =\frac{1}{2}.$ For $O\left(
n^{2}\right)  ,$ we get%
\[
-\left(  \xi_{0}+\frac{z^{2}}{4}\right)  =0\rightarrow\xi_{0}\left(  z\right)
=-\frac{z^{2}}{4}.
\]
The next term, for $O\left(  n\right)  $ gives $\xi_{1}\left(  z\right)  =0,$
and continuing this way we obtain (\ref{gn asymp}).
\end{proof}

\begin{corollary}
For $z=O\left(  1\right)  ,$ the recurrence coefficient $\gamma_{n}\left(
z\right)  ,$ satisfying the Laguerre-Freud equation (\ref{LF}) has the
asymptotic expansion%
\begin{equation}
\gamma_{n}\left(  z\right)  \sim\frac{z^{2}}{4}+\frac{z^{2}}{16}n^{-2}%
+\frac{z^{4}}{16}n^{-3}+\frac{z^{2}}{64}(1+3z^{4})n^{-4}+ O (n^{-5}),\quad
n\rightarrow\infty. \label{gan asymp}%
\end{equation}

\end{corollary}

With the previous result, we can get a first estimate for the asymptotic
behavior of $P_{n}\left(  x;z\right)  $ as $n\rightarrow\infty.$

\begin{proposition}
For $x,z=O\left(  1\right)  ,$ the polynomials $P_{n}\left(  x;z\right)  $
satisfy%
\begin{equation}
P_{n}\left(  x;z\right)  \sim\Phi_{+}^{n}\left(  x;z\right)  +\Phi_{-}%
^{n}\left(  x;z\right)  ,\quad n\rightarrow\infty, \label{Pn asymp}%
\end{equation}
where%
\[
\Phi_{\pm}\left(  x;z\right)  =\frac{x\pm\sqrt{x^{2}-z^{2}}}{2}.
\]

\end{proposition}

\begin{proof}
Using (\ref{gan asymp}) in (\ref{3-term L}), we see that \cite{MR903848}%
\[
\underset{n\rightarrow\infty}{\lim}\left[  P_{n}\left(  x;z\right)  \right]
^{\frac{1}{n}}=\Phi\left(  x;z\right)  ,
\]
where $\Phi\left(  x;z\right)  $ is a solution of the quadratic equation%
\begin{equation}
x=\Phi+\frac{z^{2}}{4}\frac{1}{\Phi}. \label{quadra}%
\end{equation}
$\allowbreak$Thus,%
\[
\Phi\left(  x;z\right)  =\frac{x\pm\sqrt{x^{2}-z^{2}}}{2},
\]
and the result follows.
\end{proof}

Note that for $x\in\left(  -z,z\right)  $%
\[
\Phi_{\pm}\left(  x;z\right)  =\frac{x\pm\mathrm{i}\sqrt{z^{2}-x^{2}}}{2},
\]
and therefore setting $x=z\cos\left(  \theta\right)  $ we have%
\[
P_{n}\left(  x;z\right)  \sim2\left(  \frac{z}{2}\right)  ^{n}\cos\left(
n\theta\right)  ,\quad n\rightarrow\infty.
\]

Since the \emph{Chebyshev polynomials of the first kind} are defined by
\cite[18.5.1]{MR2723248}%
\[
T_{n}\left(  \cos\left(  \theta\right)  \right)  =\cos\left(  n\theta\right)
,
\]
we see that%
\[
P_{n}\left(  x;z\right)  \sim2\left(  \frac{z}{2}\right)  ^{n}T_{n}\left(
\frac{x}{z}\right)  ,\quad\quad n\rightarrow\infty.
\]

\section{Structure relation and differential equation}

Semiclassical polynomials (with respect to the derivative operator) are
\emph{holonomic functions }\cite{MR2768529}, meaning that they are solutions
of a linear ODE with polynomial coefficients. In this section, we shall find
the differential equation (in $x)$ satisfied by $P_{n}\left(  x;z\right)  .$

\subsection{Structure relation}

One of the basic properties of semiclassical polynomials with respect to the
derivative operator, is a relation between $\partial_{x}P_{n}$ and $P_{n}$
\cite{MR4331433}. For the polynomials $P_{n}\left(  x;z\right)  ,$ we have the
following result.

\begin{theorem}
The polynomials $P_{n}\left(  x;z\right)  $ satisfy the
differential-recurrence relation%
\begin{equation}
\phi\left(  x;z\right)  \partial_{x}P_{n+1}=\left(  n+1\right)  P_{n+2}%
+\lambda_{n}P_{n}+\tau_{n}P_{n-2}, \label{diff-diff}%
\end{equation}
where
\begin{equation}
\lambda_{n}\left(  z\right)  =\left[  2\left(  \gamma_{n}+\gamma_{n+1}%
+\gamma_{n+2}-z^{2}-1\right)  -n\right]  \gamma_{n+1}, \label{lam}%
\end{equation}
and
\begin{equation}
\tau_{n}\left(  z\right)  =2\gamma_{n+1}\gamma_{n}\gamma_{n-1}. \label{tau}%
\end{equation}

\end{theorem}

\begin{proof}
If%
\begin{equation}
\phi\left(  x;z\right)  \partial_{x}P_{n+1}=%
{\displaystyle\sum\limits_{k=0}^{n+2}}
d_{n,k}\left(  z\right)  P_{k}, \label{diff-req1}%
\end{equation}
then using (\ref{ortho}) and (\ref{Pn odd}), we have
\begin{equation}
h_{k}d_{n,k}=L\left[  \phi\partial_{x}P_{n+1}P_{k}\right]  =0,\quad n+1\equiv
k\ \operatorname{mod}\left(  2\right)  , \label{mod}%
\end{equation}
since $L$ is symmetric and $\phi$ is an even polynomial of $x.$ Using
(\ref{Pearson2}), we get%
\begin{align*}
h_{k}d_{n,k}  &  =L\left[  \phi\partial_{x}\left(  P_{n+1}P_{k}\right)
\right]  -L\left[  \phi P_{n+1}\partial_{x}P_{k}\right] \\
&  =L\left[  2x\left(  \phi-1\right)  P_{n+1}P_{k}\right]  -L\left[  \phi
P_{n+1}\partial_{x}P_{k}\right]  .
\end{align*}
Since the polynomials $P_{n}$ are orthogonal, we conclude that $d_{n,k}=0$ for
$0\leq k<n-2,$ and because we are working with monic polynomials, we see from
(\ref{diff-req1}) that $d_{n,n+2}=n+1.$ Hence, we obtain (\ref{diff-diff}).

Using (\ref{x^2}) and (\ref{x^3}), we get%
\begin{align*}
x\left(  \phi-1\right)  P_{k}  &  =\left[  \allowbreak x^{3}-\left(
z^{2}+1\right)  x\right]  P_{k}=P_{k+3}+\gamma_{k}\gamma_{k-1}\gamma
_{k-2}P_{k-3}\\
&  +\left[  \gamma_{k}+\gamma_{k+1}+\gamma_{k+2}-\left(  z^{2}+1\right)
\right]  P_{k+1}\\
&  +\gamma_{k}\left[  \gamma_{k}+\gamma_{k+1}+\gamma_{k-1}-\left(
z^{2}+1\right)  \right]  P_{k-1}.
\end{align*}
Using (\ref{ortho}), it follows that%
\begin{align*}
L\left[  2x\left(  \phi-1\right)  P_{n+1}P_{n}\right]   &  =\left[  \gamma
_{n}+\gamma_{n+1}+\gamma_{n+2}-\left(  z^{2}+1\right)  \right]  h_{n+1},\\
L\left[  2x\left(  \phi-1\right)  P_{n+1}P_{n-2}\right]   &  =h_{n+1},
\end{align*}
and since $\phi\partial_{x}P_{k}=kP_{k+1}\left(  x\right)  +O\left(
x^{k}\right)  ,$ we see that%
\[
L\left[  \phi P_{n+1}\partial_{x}P_{n}\right]  =nh_{n+1},\quad L\left[  \phi
P_{n+1}\partial_{x}P_{n-2}\right]  =0.
\]

Hence, we conclude that%
\[
h_{n}\lambda_{n}=2\left[  \gamma_{n}+\gamma_{n+1}+\gamma_{n+2}-\left(
z^{2}+1\right)  \right]  h_{n+1}-nh_{n+1},\quad h_{n-2}\tau_{n}=2h_{n+1},
\]
and using (\ref{gamma-h}) we get%
\begin{align*}
\lambda_{n}  &  =2\left[  \left(  \gamma_{n}+\gamma_{n+1}+\gamma_{n+2}\right)
-\left(  z^{2}+1\right)  -\frac{n}{2}\right]  \gamma_{n+1},\\
\tau_{n}  &  =2\frac{h_{n+1}}{h_{n-2}}=2\gamma_{n+1}\gamma_{n}\gamma_{n-1}.
\end{align*}

\end{proof}

\subsection{Differential equation}

We will now obtain a lowering operator acting on the variable $x$ for
$P_{n}\left(  x,z\right)  .$

\begin{proposition}
Let the operator $U_{n}$ be defined by%
\[
U_{n}=A_{n}\left(  x;z\right)  \partial_{x}-B_{n}\left(  x;z\right)  ,\quad
n\in\mathbb{N},
\]
where%
\begin{equation}
A_{n}\left(  x;z\right)  =\frac{\phi\left(  x;z\right)  }{2\gamma_{n}\left(
z\right)  C_{n}\left(  x;z\right)  },\quad B_{n}\left(  x;z\right)
=\frac{n-2\gamma_{n}\left(  z\right)  }{2\gamma_{n}\left(  z\right)
C_{n}\left(  x;z\right)  }x, \label{A,B}%
\end{equation}
and%
\begin{equation}
C_{n}\left(  x;z\right)  =\phi\left(  x;z\right)  +\gamma_{n}\left(  z\right)
+\gamma_{n+1}\left(  z\right)  -n-\frac{1}{2}. \label{C}%
\end{equation}
Then,%
\begin{equation}
U_{n}P_{n}=P_{n-1},\quad n\in\mathbb{N}. \label{Un}%
\end{equation}

\end{proposition}

\begin{proof}
If we use (\ref{x^2}) in (\ref{diff-diff}), then we have%
\begin{align*}
\phi\partial_{x}P_{n+1}  &  =\left(  n+1\right)  P_{n+2}+\lambda_{n}P_{n}%
+\tau_{n}\frac{\left[  x^{2}-\left(  \gamma_{n}+\gamma_{n+1}\right)  \right]
P_{n}-P_{n+2}}{\gamma_{n}\gamma_{n-1}}\\
&  =\left(  n+1-\frac{\tau_{n}}{\gamma_{n}\gamma_{n-1}}\right)  P_{n+2}%
+\left[  \lambda_{n}+\frac{\tau_{n}\left(  x^{2}-\gamma_{n}-\gamma
_{n+1}\right)  }{\gamma_{n}\gamma_{n-1}}\right]  P_{n}.
\end{align*}
From (\ref{3-term L}), we see that
\begin{equation}
\left[  \phi\partial_{x}-\left(  n+1-\frac{\tau_{n}}{\gamma_{n}\gamma_{n-1}%
}\right)  x\right]  P_{n+1}=\left[  \lambda_{n}+\tau_{n}\frac{x^{2}-\gamma
_{n}}{\gamma_{n}\gamma_{n-1}}-\left(  n+1\right)  \gamma_{n+1}\right]  P_{n}.
\label{dd1}%
\end{equation}
Using (\ref{lam}) and (\ref{tau}) in (\ref{dd1}), we get%
\[
\left[  \phi\partial_{x}-\left(  n+1-2\gamma_{n+1}\right)  x\right]
P_{n+1}=2\gamma_{n+1}\left(  \phi+\gamma_{n+1}+\gamma_{n+2}-n-\frac{3}%
{2}\right)  P_{n},
\]
and the result follows.
\end{proof}

Taking into account the lowering operator $U_{n},$ we can deduce a second
order linear differential equation (in $x)$ for $P_{n}\left(  x,z\right)  .$

\begin{theorem}
Let the differential operator $D_{n}$ be defined by%
\begin{equation}%
\begin{tabular}
[c]{l}%
$D_{n}=\phi^{2}C_{n}\partial_{x}^{2}-2x\phi\left[  \left(  \phi-1\right)
C_{n}+\phi\right]  \partial_{x}$\\
$+\left(  n-2\gamma_{n}\right)  \left[  2x^{2}\phi-\left(  \phi-2x^{2}%
\phi+nx^{2}-2x^{2}\gamma_{n}\right)  C_{n}\right]  +4\gamma_{n}C_{n-1}%
C_{n}^{2}.$%
\end{tabular}
\ \label{Vn}%
\end{equation}
Then, $D_{n}P_{n}=0$ for all $n\in\mathbb{N}.$
\end{theorem}

\begin{proof}
Using (\ref{Un}) in (\ref{3-term L}), we get%
\[
xU_{n}P_{n}=P_{n}+\gamma_{n-1}U_{n-1}U_{n}P_{n}.
\]
If $y$ is a function of $x,$ we have%
\begin{align*}
U_{n-1}U_{n}y  &  =\left(  A_{n-1}\partial_{x}-B_{n-1}\right)  \left(
A_{n}y^{\prime}-B_{n}y\right) \\
&  =A_{n-1}\left(  \partial_{x}A_{n}y^{\prime}+A_{n}y^{\prime\prime}%
-\partial_{x}B_{n}y-B_{n}y^{\prime}\right)  -B_{n-1}\left(  A_{n}y^{\prime
}-B_{n}y\right)  ,
\end{align*}
and therefore%
\begin{gather*}
\left(  \gamma_{n-1}U_{n-1}U_{n}-xU_{n}+1\right)  y=\gamma_{n-1}A_{n-1}%
A_{n}y^{\prime\prime}\\
+\gamma_{n-1}\left[  A_{n-1}\left(  \partial_{x}A_{n}-B_{n}\right)
-A_{n}B_{n-1}\right]  y^{\prime}-xA_{n}y^{\prime}\\
+\gamma_{n-1}\left(  B_{n}B_{n-1}-A_{n-1}\partial_{x}B_{n}\right)
y+xB_{n}y+y.
\end{gather*}

From (\ref{A,B}) and (\ref{C}), we see that%
\[
\gamma_{n-1}A_{n-1}A_{n}=\frac{\phi^{2}}{4\gamma_{n}C_{n-1}C_{n}},
\]%
\[
\gamma_{n-1}\left[  A_{n-1}\left(  \partial_{x}A_{n}-B_{n}\right)
-A_{n}B_{n-1}\right]  -xA_{n}=-\frac{\left(  \phi-1\right)  C_{n}+\phi
}{2\gamma_{n}C_{n-1}C_{n}^{2}}x\phi,
\]
and%
\begin{align*}
&  \gamma_{n-1}\left(  B_{n}B_{n-1}-A_{n-1}\partial_{x}B_{n}\right)
+xB_{n}+1\\
&  =\frac{\left(  n-2\gamma_{n}\right)  \left[  2x^{2}\phi-\left(  \phi
+nx^{2}-2x^{2}\phi-2x^{2}\gamma_{n}\right)  C_{n}\right]  }{4\gamma_{n}%
C_{n-1}C_{n}^{2}}+1.
\end{align*}
Multiplying by $4\gamma_{n}C_{n-1}C_{n}^{2},$ the result follows.
\end{proof}

\begin{remark}
We can write the third term in the differential operator $D_{n}$ as
\[
\allowbreak\left(  n-2\gamma_{n}\right)  \left[  \left(  2x^{2}-1\right)
C_{n}+2x^{2}\right]  \phi+\left[  4\gamma_{n}C_{n-1}C_{n}-\left(
n-2\gamma_{n}\right)  ^{2}x^{2}\right]  C_{n}%
\]
and since from (\ref{C}) we see that $C_{n}\left(  x;z\right)  =\phi\left(
x;z\right)  +l_{n}\left(  z\right)  ,$ we have%
\[
4\gamma_{n}C_{n-1}C_{n}-\left(  n-2\gamma_{n}\right)  ^{2}x^{2}=4\gamma
_{n}\left(  \phi+l_{n-1}\right)  \left(  \phi+l_{n}\right)  -\left(
n-2\gamma_{n}\right)  ^{2}x^{2},
\]
where%
\begin{equation}
l_{n}\left(  z\right)  =\gamma_{n}\left(  z\right)  +\gamma_{n+1}\left(
z\right)  -n-\frac{1}{2}. \label{ln}%
\end{equation}
The Laguerre-Freud equation (\ref{LF}) can be written as%
\[
4\gamma_{n}l_{n}l_{n-1}=z^{2}\left(  n-2\gamma_{n}\right)  ^{2},
\]
and therefore%
\begin{gather*}
4\gamma_{n}\left(  \phi+l_{n-1}\right)  \left(  \phi+l_{n}\right)  -\left(
n-2\gamma_{n}\right)  ^{2}x^{2}=4\gamma_{n}\left(  \phi+l_{n}+l_{n-1}%
-\frac{l_{n}l_{n-1}}{z^{2}}\right)  \phi\\
=\left[  4\gamma_{n}\left(  C_{n}+l_{n-1}\right)  -\left(  n-2\gamma
_{n}\right)  ^{2}\right]  \phi.
\end{gather*}

As a consequence, we can write the differential equation for $P_{n}\left(
x;z\right)  $ in the reduced form%
\begin{equation}%
\begin{tabular}
[c]{l}%
$\phi C_{n}\partial_{x}^{2}P_{n}-2x\left[  \left(  \phi-1\right)  C_{n}%
+\phi\right]  \partial_{x}P_{n}+\left(  n-2\gamma_{n}\right)  \left[  \left(
2x^{2}-1\right)  C_{n}+2x^{2}\right]  P_{n}$\\
$+\left[  4\gamma_{n}\left(  C_{n}+l_{n-1}\right)  -\left(  n-2\gamma
_{n}\right)  ^{2}\right]  C_{n}P_{n}=0.$%
\end{tabular}
\label{ODE}%
\end{equation}

\end{remark}

Using (\ref{gan asymp}) in (\ref{Vn}), we get
\[
D_{n}\sim-n\phi^{2}\partial_{x}^{2}+2nx\phi\left(  \phi-1\right)  \partial
_{x}+n^{3}\phi,\quad n\rightarrow\infty,
\]
and therefore we see that if $P_{n}\left(  x;z\right)  \sim\Phi^{n}\left(
x;z\right)  $ as $n\rightarrow\infty,$ then to leading order
\[
1-\phi\left(  \frac{\partial_{x}\Phi}{\Phi}\right)  ^{2}=0.
\]
The solutions of this Riccati equation are%
\[
\Phi_{\pm}\left(  x;z\right)  =\frac{x\pm\sqrt{\phi\left(  x;z\right)  }}{2},
\]
in agreement with (\ref{Pn asymp}).$\allowbreak$

\section{Electrostatic interpretation of the zeros}

It is very well known that the zeros of orthogonal polynomials with respect to
a positive definite linear functional are real, simple, and located in the
interior of the convex hull of the support of the linear functional
\cite{MR0481884}. Thus, let denote by $\left\{  x_{n,k}\left(  z\right)
\right\}  _{1\leq k\leq n}$ the zeros of $P_{n}\left(  x;z\right)  $ in an
increasing order, i.e.
\[
P_{n}\left(  x_{n,k};z\right)  =0,\quad1\leq k\leq n,
\]

\noindent and $x_{n,1} < x_{n,2} < \cdots< x_{n,n}.$

Evaluating the operator $D_{n}$ at $x=x_{n,k},$ we see that%
\[
\left[  \frac{\partial_{x}^{2}P_{n}}{\partial_{x}P_{n}}\right]  _{x=x_{n,k}%
}-2x_{n,k}\frac{\left[  \phi\left(  x_{n,k};z\right)  -1\right]  C_{n}\left(
x_{n,k};z\right)  +\phi\left(  x_{n,k};z\right)  }{\phi\left(  x_{n,k}%
;z\right)  C_{n}\left(  x_{n,k};z\right)  }=0,
\]
or, using (\ref{C}),%
\begin{equation}
\left[  \frac{\partial_{x}^{2}P_{n}}{\partial_{x}P_{n}}\right]  _{x=x_{n,k}%
}-2x_{n,k}+\frac{1}{x_{n,k}-z}+\frac{1}{x_{n,k}+z}-\frac{1}{x_{n,k}-\zeta
_{n}(z)}-\frac{1}{x_{n,k}+\zeta_{n}(z)}=0, \label{Poten1}%
\end{equation}
where%
\begin{equation}
\zeta_{n}^{2}\left(  z\right)  =z^{2}+n+\frac{1}{2}-\gamma_{n}-\gamma_{n+1}.
\label{zetan}%
\end{equation}

Using (\ref{zetan}) in (\ref{LF}), we get%
\[
\gamma_{n}\left(  \zeta_{n}^{2}(z)-z^{2}\right)  \left(  \zeta_{n-1}%
^{2}(z)-z^{2}\right)  =z^{2}\left(  \frac{n}{2}-\gamma_{n}\right)  ^{2},
\]
and since
\[
\zeta_{0}^{2}(z) -z^{2}=\frac{1}{2}-\frac{\mu_{1}}{\mu_{0}}=\frac{ze^{-z^{2}}%
}{\sqrt{\pi}\operatorname{erf}\left(  z\right)  }>0,\quad z\in\mathbb{R},
\]
it follows by induction that $\zeta_{n}^{2}(z) -z^{2}>0$ for all
$n\in\mathbb{N}_{0}.$ In fact, one can show that%
\[
\underset{z\in\mathbb{R}}{\min}\zeta_{n}^{2}\left(  z\right)  =\zeta_{n}%
^{2}\left(  0\right)  =n+\frac{1}{2},\quad n\in\mathbb{N}_{0}.
\]

Using (\ref{gan asymp}) in (\ref{zetan}), we obtain%
\[
\zeta_{n}^{2}\left(  z\right)  \sim n+\frac{z^{2}+1}{2}-\frac{z^{2}}{8}%
n^{-2}+\frac{z^{2}\left(  1-z^{2}\right)  }{8}n^{-3},\quad n\rightarrow
\infty,
\]
and therefore%
\[
\zeta_{n}\left(  z\right)  -z=\sqrt{n}-z+O\left(  n^{-\frac{1}{2}}\right)
,\quad n\rightarrow\infty.
\]
Thus, for $z=O\left(  1\right)  $, the points $\pm\zeta_{n}\left(  z\right)  $
are outside the interval $\left[  -z,z\right]  ,$ and "moving" outwards as
$n\rightarrow\infty.$

Using the previous results, we have shown the following theorem.

\begin{theorem}
The zeros of $P_{n}\left(  x;z\right)  $ are located at the equilibrium points
of $n$ unit charged particles located in the interval $(-z, z)$ under the
influence of the potential%
\[
V_{n}\left(  x;z\right)  =x^{2}-\ln\left\vert x^{2}-z^{2}\right\vert
+\ln\left\vert x^{2}-\zeta_{n}^{2}\left(  z\right)  \right\vert .
\]

\end{theorem}

\begin{proof}
As it's well known, if we write
\[
P_{n}\left(  x;z\right)  =%
{\displaystyle\prod\limits_{k=1}^{n}}
\left(  x-x_{n,k}\right)  ,
\]
then \cite[Chapter 10]{MR4331433}%
\[
\left[  \frac{\partial_{x}^{2}P_{n}}{\partial_{x}P_{n}}\right]  _{x=x_{n,k}%
}=\sum_{\substack{j=1\\j\neq k}}^{n}\frac{2}{x_{n,k}-x_{n,j}},
\]
and therefore (\ref{Poten1}) gives%
\[
\sum_{\substack{j=1\\j\neq k}}^{n}\frac{2}{x_{n,j}-x_{n,k}}+2x_{n,k}-\frac
{1}{x_{n,k}-z}-\frac{1}{x_{n,k}+z}+\frac{1}{x_{n,k}-\zeta_{n}}+\frac
{1}{x_{n,k}+\zeta_{n}}=0,
\]
or equivalently
\[
\frac{\partial E}{\partial x_{n,k}}=0,\quad1\leq k\leq n,
\]
where the total energy of the system is%
\[
E\left(  x_{n,1},\ldots,x_{n,n}\right)  =-2\sum_{1\leq j<k\leq n}^{n}%
\ln\left\vert x_{n,k}-x_{n,j}\right\vert +\sum_{k=1}^{n}x_{n,k}^{2}%
-\ln\left\vert x_{n,k}^{2}-z^{2}\right\vert +\ln\left\vert x_{n,k}^{2}%
-\zeta_{n}^{2}\right\vert .
\]
It follows that the external potential is%
\[
V_{n}\left(  x;z\right)  =x^{2}-\ln\left\vert x^{2}-z^{2}\right\vert
+\ln\left\vert x^{2}-\zeta_{n}^{2}\right\vert .
\]

\end{proof}

\section{Toda-type behavior}

Differentiating (\ref{L}) with respect to $z,$ we have
\begin{equation}
\partial_{z}L\left[  p\left(  x;z\right)  \right]  =e^{-z^{2}}\left[  p\left(
z;z\right)  +p\left(  -z;z\right)  \right]  +L\left[  \partial_{z}p\left(
x;z\right)  \right]  , \label{Dz L}%
\end{equation}
and we note that%
\begin{equation}
\partial_{z}L\left[  \phi p\right]  =L\left[  \partial_{z}\left(  \phi
p\right)  \right]  , \label{Dz L phi}%
\end{equation}
where $\phi\left(  x;z\right)  $ was defined in (\ref{phi}).

In particular,%
\[
u_{n}^{\prime}=\partial_{z}L\left[  x^{2n}\right]  =2z^{2n}e^{-z^{2}},
\]
and using (\ref{mu req}), we get
\begin{equation}
zu_{n}^{\prime}=2z^{2n+1}e^{-z^{2}}=\left(  2n+1\right)  u_{n}-2u_{n+1}.
\label{Dz un}%
\end{equation}
On the other hand, using (\ref{Dz L phi}) we have%
\[
u_{n+1}^{\prime}-2zu_{n}-z^{2}u_{n}^{\prime}=\partial_{z}\left(  u_{n+1}%
-z^{2}u_{n}\right)  =-2zu_{n},
\]
and therefore%
\begin{equation}
u_{n+1}^{\prime}=z^{2}u_{n}^{\prime},\quad n\geq0. \label{Dz un 1}%
\end{equation}

Using the differential-recurrence for the moments (\ref{Dz un 1}), we can
obtain a first order ODE (in $z)$ for the Stieltjes function $S(t;z).$

\begin{proposition}
Let the function $S(t;z)$ be defined by (\ref{S1}), and $\phi\left(
x;z\right)  $ be defined by (\ref{phi}). Then,%
\begin{equation}
\phi\left(  t;z\right)  \partial_{z}S=2te^{-z^{2}}. \label{Sz}%
\end{equation}

\end{proposition}

\begin{proof}
Using (\ref{u n+k}), we have%
\[%
{\displaystyle\sum\limits_{n\geq0}}
\frac{u_{n+1}}{t^{2n+1}}=t^{2}S-tu_{0},
\]
and therefore (\ref{Dz un 1}) gives
\[
t^{2}\partial_{z}S-tu_{0}^{\prime}=z^{2}\partial_{z}S.
\]
Using (\ref{mu0}), the result follows.
\end{proof}

\begin{remark}
Clearly, (\ref{Sz}) and the initial condition $S\left(  t;0\right)  =0$ yield%
\[
S(t;z)=2t%
{\displaystyle\int\limits_{0}^{z}}
\frac{e^{-x^{2}}}{t^{2}-x^{2}}dx,
\]
which is just the definition (\ref{S}), since%
\[%
{\displaystyle\sum\limits_{n\geq0}}
\frac{x^{2n}}{t^{2n+1}}=\frac{t}{t^{2}-x^{2}}.
\]

\end{remark}

Next, we will obtain some equations relating $h_{n}^{\prime},\gamma
_{n}^{\prime}$ with $h_{n},\gamma_{n}.$

\begin{theorem}
The functions $h_{n}\left(  z\right)  ,\gamma_{n}\left(  z\right)  $ satisfy
the Toda-type equations%
\begin{equation}
\vartheta\ln\left(  h_{n}\right)  =2n+1-2\left(  \gamma_{n+1}+\gamma
_{n}\right)  , \label{D hn}%
\end{equation}
and%
\begin{equation}
\vartheta\ln\left(  \gamma_{n}\right)  =2\left(  \gamma_{n-1}-\gamma
_{n+1}+1\right)  , \label{D gn}%
\end{equation}
where $\vartheta$ is the operator defined by%
\begin{equation}
\vartheta=z\partial_{z}. \label{theta}%
\end{equation}

\end{theorem}

\begin{proof}
Using (\ref{Dz L}), we have%
\[
h_{n}^{\prime}=\partial_{z}L\left[  P_{n}^{2}\right]  =2e^{-z^{2}}P_{n}%
^{2}\left(  z;z\right)  +L\left[  2P_{n}\partial_{z}P_{n}\right]  ,
\]
but since $\deg\left(  \partial_{z}P_{n}\right)  \leq n-2$ we can use
(\ref{ortho}) and (\ref{LF6}) and obtain%
\[
zh_{n}^{\prime}=\left(  2n+1\right)  h_{n}-2\left(  h_{n+1}+\gamma_{n}%
^{2}h_{n-1}\right)  ,
\]
or, using (\ref{gamma-h})%
\[
z\frac{h_{n}^{\prime}}{h_{n}}=2n+1-2\left(  \gamma_{n+1}+\gamma_{n}\right)  .
\]

Note that (\ref{gamma-h}) gives%
\[
\frac{\gamma_{n}^{\prime}}{\gamma_{n}}=\frac{h_{n}^{\prime}}{h_{n}}%
-\frac{h_{n-1}^{\prime}}{h_{n-1}},
\]
and hence we can write (\ref{D hn}) in terms of $\gamma_{n}$%
\[
z\frac{\gamma_{n}^{\prime}}{\gamma_{n}}=\allowbreak2\left(  \gamma
_{n-1}-\gamma_{n+1}+1\right)  .
\]

\end{proof}

If we combine the Laguerre-Freud equation for $\gamma_{n}$ (\ref{LF}) and the
Toda-type equation for $h_{n}\left(  z\right)  ,$ we obtain the following result.

\begin{proposition}
The function $h_{n}\left(  z\right)  $ satisfies
\[
h_{n}^{\prime}h_{n-1}^{\prime}=\left(  nh_{n-1}-2h_{n}\right)  ^{2},\quad
n\geq1.
\]

\end{proposition}

\begin{proof}
From (\ref{D hn}), we see that%
\[
\frac{1}{2}\vartheta\ln\left(  h_{n}\right)  =n+\frac{1}{2}-\left(
\gamma_{n+1}+\gamma_{n}\right)  ,
\]
and using this in (\ref{LF}), we get%
\[
\frac{1}{2}\vartheta\ln\left(  h_{n}\right)  \frac{1}{2}\vartheta\ln\left(
h_{n-1}\right)  =\frac{z^{2}}{\gamma_{n}}\left(  \frac{n}{2}-\gamma
_{n}\right)  ^{2},
\]
or using (\ref{gamma-h}) and (\ref{theta})%
\[
h_{n}\partial_{z}\ln\left(  h_{n}\right)  h_{n-1}\partial_{z}\ln\left(
h_{n-1}\right)  =\left(  nh_{n-1}-2h_{n}\right)  ^{2},
\]
and the result follows.
\end{proof}

\subsection{Nonlinear ODE}

Using the Laguerre-Freud equations (\ref{LaguerreFreud1}), (\ref{LF}) and the
differential-recurrence relation (\ref{D gn}), we can derive a nonlinear
second order ODE for $\gamma_{n}\left(  z\right)  .$

\begin{theorem}
The function $\gamma_{n}\left(  z\right)  $ satisfies
\begin{equation}
z^{2}\left[  \gamma_{n}^{\prime\prime}+2\left(  6\gamma_{n}-n\right)  \left(
2\gamma_{n}-n\right)  \right]  ^{2}=4\left(  z^{2}+2\gamma_{n}-n\right)
^{2}\left[  \left(  \gamma_{n}^{\prime}\right)  ^{2}+4\gamma_{n}\left(
2\gamma_{n}-n\right)  ^{2}\right]  . \label{Painleve}%
\end{equation}

\end{theorem}

\begin{proof}
Setting $n\rightarrow n-1$ in (\ref{LF}), we have%
\[
\gamma_{n-1}\left(  n-\frac{1}{2}-\gamma_{n-1}-\gamma_{n}\right)  \left(
n-\frac{3}{2}-\gamma_{n-1}-\gamma_{n-2}\right)  =z^{2}\left(  \frac{n-1}%
{2}-\gamma_{n-1}\right)  ^{2},
\]
and solving for $\gamma_{n-2},$ we obtain%
\begin{equation}
\gamma_{n-2}=\frac{z^{2}\left(  \frac{1}{2}-\frac{1}{2}n+\gamma_{n-1}\right)
^{2}}{\gamma_{n-1}\left(  \frac{1}{2}-n+\gamma_{n}+\gamma_{n-1}\right)
}+n-\frac{3}{2}-\gamma_{n-1}. \label{gm2}%
\end{equation}
Solving for $\gamma_{n+2}$ in (\ref{LaguerreFreud1}), we get%
\begin{equation}
\gamma_{n+2}=\frac{\gamma_{n}\left(  \gamma_{n}+\gamma_{n-1}-z^{2}-n+\frac
{1}{2}\right)  +\gamma_{n+1}\left(  n+\frac{3}{2}-\gamma_{n+1}+z^{2}\right)
-\frac{1}{2}z^{2}}{\gamma_{n+1}}. \label{gp2}%
\end{equation}
From (\ref{gm2}) and (\ref{gp2}) we conclude that%
\begin{equation}%
\begin{tabular}
[c]{l}%
$\gamma_{n-1}\gamma_{n-2}+\gamma_{n+1}\gamma_{n+2}=\frac{z^{2}\left(  \frac
{1}{2}-\frac{1}{2}n+\gamma_{n-1}\right)  ^{2}}{\frac{1}{2}-n+\gamma_{n}%
+\gamma_{n-1}}+\left(  n-\frac{3}{2}-\gamma_{n-1}\right)  \gamma_{n-1}$\\
$+\gamma_{n}\left(  \gamma_{n}+\gamma_{n-1}-z^{2}-n+\frac{1}{2}\right)
+\gamma_{n+1}\left(  n+\frac{3}{2}-\gamma_{n+1}+z^{2}\right)  -\frac{1}%
{2}z^{2}.$%
\end{tabular}
\label{gcomb}%
\end{equation}

Differentiating (\ref{D gn}) with respect to $z,$ we see that%
\[
\frac{\left(  z\gamma_{n}^{\prime}\right)  ^{\prime}}{2}=\gamma_{n}^{\prime
}\left(  \gamma_{n-1}-\gamma_{n+1}+1\right)  +\gamma_{n}\left(  \gamma
_{n-1}^{\prime}-\gamma_{n+1}^{\prime}\right)
\]
and using (\ref{D gn}) again we have%
\begin{equation}
\left(  z\gamma_{n}^{\prime}\right)  ^{\prime}=\frac{z\left(  \gamma
_{n}^{\prime}\right)  ^{2}}{\gamma_{n}}+\frac{4}{z}\gamma_{n}\left[
\gamma_{n-1}\left(  \gamma_{n-2}-\gamma_{n}+1\right)  -\gamma_{n+1}\left(
\gamma_{n}-\gamma_{n+2}+1\right)  \right]  . \label{dg2}%
\end{equation}
Using (\ref{gcomb}) in (\ref{dg2}), we get%
\begin{gather}
\frac{z}{4\gamma_{n}}\left(  z\gamma_{n}^{\prime}\right)  ^{\prime}=\left(
\frac{z\gamma_{n}^{\prime}}{2\gamma_{n}}\right)  ^{2}+\allowbreak\frac
{z^{2}\left(  \frac{1}{2}-\frac{1}{2}n+\gamma_{n-1}\right)  ^{2}}{\frac{1}%
{2}-n+\gamma_{n}+\gamma_{n-1}}\nonumber\\
+\gamma_{n}\left(  \frac{1}{2}-n-z^{2}+\gamma_{n}-\gamma_{n+1}\right)
-\frac{1}{2}z^{2}\label{non1}\\
+\gamma_{n+1}\left(  \allowbreak z^{2}+\allowbreak n+\frac{1}{2}-\gamma
_{n+1}\right)  +\gamma_{n-1}\left(  n-\frac{1}{2}-\gamma_{n-1}\right)
.\nonumber
\end{gather}

Note that (\ref{LF}), (\ref{D gn}), and (\ref{non1}) are three equations
relating $\gamma_{n}^{\prime\prime},\gamma_{n}^{\prime},\gamma_{n}%
,\gamma_{n-1}$ and $\gamma_{n+1}.$ Thus, $\gamma_{n+1},\gamma_{n-1}$ can be
eliminated from the system and we obtain (\ref{Painleve}).
\end{proof}

\subsection{Power series}

Using the nonlinear recurrence (\ref{LF}), we can see that%

\begin{equation}
\gamma_{n}\left(  z\right)  =\frac{n^{2}z^{2}}{4n^{2}-1}+\frac{4n^{3}z^{4}%
}{\left(  4n^{2}-1\right)  ^{2}\left(  4n^{2}-9\right)  }+O\left(
z^{6}\right)  ,\quad z\rightarrow0. \label{gn z=0}%
\end{equation}
To obtain higher order terms, we can use (\ref{D gn}).

\begin{theorem}
The Maclaurin series of the function $\gamma_{n}\left(  z\right)  $ is%
\begin{equation}
\gamma_{n}\left(  z\right)  =%
{\displaystyle\sum\limits_{k=1}^{\infty}}
\eta_{n,k}z^{2k}, \label{gn eta}%
\end{equation}
where
\begin{equation}
\eta_{n,1}=\frac{n^{2}}{4n^{2}-1}, \label{eta1}%
\end{equation}
and
\begin{equation}
\eta_{n,k}=\frac{1}{k-1}%
{\displaystyle\sum\limits_{j=1}^{k-1}}
\left(  \eta_{n-1,j}-\eta_{n+1,j}\right)  \eta_{n,k-j},\quad k\geq2.
\label{eta}%
\end{equation}

\end{theorem}

\begin{proof}
The first term (\ref{eta1}) follows immediately from (\ref{gn z=0}). From
(\ref{gn eta}) we have%
\begin{equation}
\frac{z}{2}\gamma_{n}^{\prime}\left(  z\right)  =%
{\displaystyle\sum\limits_{k=1}^{\infty}}
k\eta_{n,k}z^{2k}, \label{left gn}%
\end{equation}
and%
\[
\gamma_{n-1}\left(  z\right)  -\gamma_{n+1}\left(  z\right)  =%
{\displaystyle\sum\limits_{k=1}^{\infty}}
\left(  \eta_{n-1,k}-\eta_{n+1,k}\right)  z^{2k}.
\]
Using the Cauchy product, we get
\begin{equation}
\left[  z^{2k}\right]  \left(  \gamma_{n-1}-\gamma_{n+1}+1\right)  \gamma
_{n}=\eta_{n,k}+%
{\displaystyle\sum\limits_{j=1}^{k-1}}
\left(  \eta_{n-1,j}-\eta_{n+1,j}\right)  \eta_{n,k-j}, \label{right gn}%
\end{equation}
where $\left[  z^{m}\right]  $ denotes the coefficient of $z^{m}$ in the given expression.

Using (\ref{left gn}) and (\ref{right gn}) in (\ref{D gn}), we obtain
\[
k\eta_{n,k}=\eta_{n,k}+%
{\displaystyle\sum\limits_{j=1}^{k-1}}
\left(  \eta_{n-1,j}-\eta_{n+1,j}\right)  \eta_{n,k-j},
\]
and (\ref{eta}) follows.
\end{proof}

\begin{remark}
If we introduce the forward and backward difference operators
\begin{equation}
\Delta f\left(  n\right)  =f\left(  n+1\right)  -f\left(  n\right)
,\quad\nabla f\left(  n\right)  =f\left(  n\right)  -f\left(  n-1\right)  ,
\label{diff}%
\end{equation}
then we can write (\ref{eta}) as
\[
\eta_{n,k}=-\frac{1}{k-1}%
{\displaystyle\sum\limits_{j=1}^{k-1}}
\left(  \Delta+\nabla\right)  \eta_{n,j}\eta_{n,k-j},\quad k\geq2,
\]
and it follows that (\ref{gn eta}) is an asymptotic series as $n\rightarrow
\infty.$ Using the same methods that we introduced in \cite{MR4136730}, we can
show that for $k\geq2$%
\[
\eta_{n,k}=O\left(  n^{-k-1}\right)  ,\quad n\rightarrow\infty.
\]

\end{remark}

Using (\ref{3-term L}) and (\ref{gn eta}), we see that%
\[
P_{n}\left(  x;z\right)  =x^{n}-\frac{n\left(  n-1\right)  }{2\left(
2n-1\right)  }x^{n-2}z^{2}+O\left(  z^{4}\right)  .
\]
We obtain higher order terms in the following theorem.

\begin{theorem}
The Maclaurin series of the polynomials $P_{n}\left(  x;z\right)  $ is%
\begin{equation}
P_{n}\left(  x;z\right)  =x^{n}+%
{\displaystyle\sum\limits_{k=1}^{\infty}}
\alpha_{n,k}\left(  x\right)  z^{2k}, \label{Pn series}%
\end{equation}
where
\[
\alpha_{0,k}\left(  x\right)  =\alpha_{1,k}\left(  x\right)  =0,\quad k\geq1,
\]
and the coefficients $\alpha_{n,k}\left(  x\right)  $ satisfy the recurrence%
\begin{equation}
\alpha_{n+1,k}\left(  x\right)  -x\alpha_{n,k}\left(  x\right)  +\eta
_{n,k}x^{n-1}+%
{\displaystyle\sum\limits_{j=1}^{k-1}}
\alpha_{n-1,j}\left(  x\right)  \eta_{n,k-j}=0. \label{alpha req}%
\end{equation}

In particular,%
\begin{equation}
\alpha_{n,1}\left(  x\right)  =-\frac{n\left(  n-1\right)  }{2\left(
2n-1\right)  }x^{n-2}, \label{alpha1}%
\end{equation}
and
\begin{equation}
\alpha_{n,2}\left(  x\right)  =n\left(  n-1\right)  \frac{8n\left(
n-1\right)  x^{2}+\left(  2n+1\right)  \left(  2n-1\right)  \left(
n-2\right)  \left(  n-3\right)  }{8\left(  2n+1\right)  \left(  2n-1\right)
^{2}\left(  2n-3\right)  }x^{n-4}. \label{alpha 2}%
\end{equation}

\end{theorem}

\begin{proof}
Using (\ref{gn eta}) and (\ref{Pn series}) in (\ref{3-term L}), we have%
\begin{gather*}
x^{n+1}+%
{\displaystyle\sum\limits_{k=1}^{\infty}}
x\alpha_{n,k}\left(  x\right)  z^{2k}=x^{n+1}+%
{\displaystyle\sum\limits_{k=1}^{\infty}}
\alpha_{n+1,k}\left(  x\right)  z^{2k}\\
+x^{n-1}\gamma_{n}\left(  z\right)  +\gamma_{n}\left(  z\right)  \left(
{\displaystyle\sum\limits_{k=1}^{\infty}}
\alpha_{n-1,k}\left(  x\right)  z^{2k}\right)  ,
\end{gather*}
and therefore we obtain the recurrence
\[
x\alpha_{n,k}\left(  x\right)  =\alpha_{n+1,k}\left(  x\right)  +\eta
_{n,k}x^{n-1}+%
{\displaystyle\sum\limits_{j=1}^{k-1}}
\alpha_{n-1,j}\left(  x\right)  \eta_{n,k-j}.
\]

If $k=1,$ then (\ref{alpha req}) becomes%
\[
\alpha_{n+1,1}\left(  x\right)  -x\alpha_{n,1}\left(  x\right)  =-\eta
_{n,1}x^{n-1},
\]
and the solution with initial condition $\alpha_{0,k}\left(  x\right)  =0$ is%
\[
\alpha_{n,1}\left(  x\right)  =-x^{n-2}%
{\displaystyle\sum\limits_{i=0}^{n-1}}
\eta_{i,1}\left(  x\right)  =-\frac{n\left(  n-1\right)  }{2\left(
2n-1\right)  }x^{n-2}.
\]

Setting $k=2$ in (\ref{alpha req}) we get%
\[
\alpha_{n+1,2}\left(  x\right)  -x\alpha_{n,2}\left(  x\right)  =-\eta
_{n,1}\alpha_{n-1,1}\left(  x\right)  -\eta_{n,2}x^{n-1},
\]
and therefore%
\[
\alpha_{n,2}\left(  x\right)  =-x^{n-1}%
{\displaystyle\sum\limits_{i=0}^{n-1}}
\left[  x^{-i}\eta_{i,1}\alpha_{i-1,1}\left(  x\right)  +\frac{\eta_{i,2}}%
{x}\right]  .
\]
Using (\ref{eta}) and (\ref{alpha1}), we obtain (\ref{alpha 2}).
\end{proof}

\section{Conclusions}

We have defined the family of truncated Hermite polynomials $P_{n}\left(
x;z\right)  $, orthogonal with respect to the linear functional
\[
L\left[  p\right]  =%
{\displaystyle\int\limits_{-z}^{z}}
p\left(  x\right)  e^{-x^{2}}dx,\quad p\in\mathbb{R}\left[  x\right]  ,\quad
z>0.
\]
Such a linear functional satisfies the Pearson equation%
\[
L\left[  \phi\partial_{x}p\right]  =L\left[  2x\left(  \phi-1\right)
p\right]  ,\quad\phi\left(  x;z\right)  =x^{2}-z^{2}.
\]
We related $P_{n}\left(  x;z\right)  $ to the Hermite and Rys polynomials, and
studied the sequence $P_{n}\left(  x;z\right)  $ as semiclassical polynomials
of class $2$. The expansion of $\phi\partial_{x}P_{n}$ in the $\left\{
P_{k}\right\}  _{k\geq0}$ basis (structure relation), an asymptotic
approximation (for large $n)$, a lowering operator, second order ODE (in $x),$
and power series (in $z)$ for $P_{n}\left(  x;z\right)  $ are given.

We obtained a second order linear recurrence for the moments, as well as a
differential-recurrence equation in terms of the variable $z$. An asymptotic
approximation for the moments (as $z\rightarrow\infty)$ is obtained.
Differential equations (in $t$ and $z$, respectively) for the Stieltjes
function $S\left(  t;z\right)  $ of the moments associated with $L$ are deduced.

We also got nonlinear recurrences (Laguerre-Freud equations) and a nonlinear
ODE that the parameters $\gamma_{n}\left(  z\right)  $ satisfy. As a
consequence, an asymptotic approximation (for large $n)$, a
differential-recurrence equation, and a power series for the coefficients
$\gamma_{n}\left(  z\right)  $ in the recurrence relation of $P_{n}\left(
x;z\right)  $ are obtained.

We plan to continue our research on these polynomials in order to obtain
asymptotic expansions for $P_{n}\left(  x;z\right)  $ as $n\rightarrow\infty,$
$z\rightarrow\infty$ as well as when both $n,z\rightarrow\infty$
simultaneously. One should be able to obtain the well known asymptotic
approximations for the Hermite polynomials in the last case.

Finally, we will deal with the analysis of truncated Laguerre polynomials, as
well as other families of truncated semiclassical polynomials.

\bigskip

\textbf{Acknowledgement}

The work of the first author was supported by the strategic program
"Innovatives \"{O}-- 2010 plus" from the Upper Austrian Government, and by the
grant SFB F50 (F5009-N15) from the Austrian Science Foundation (FWF). We thank
Prof. Carsten Schneider for his generous sponsorship.

The work of the second author has been supported by FEDER/Ministerio de
Ciencia e Innovaci\^{o}n-Agencia Estatal de Investigaci\'{o}n of Spain, grant
PGC2018-096504-B-C33, and the Madrid Government (Comunidad de Madrid-Spain)
under the Multiannual Agreement with UC3M in the line of Excellence of
University Professors, grant EPUC3M23 in the context of the V PRICIT (Regional
Programme of Research and Technological Innovation).

\newif\ifabfull\abfullfalse\input apreambl


\begin{thebibliography}{99}                                                                                               %


\bibitem {PMID:26613300}%
\abtype{0}{A.~Asadchev, V.~Allada, J.~Felder, B.~M. Bode, M.~S.
Gordon\abphrase{1}\abphrase{0}T.~L. Windus}. \newblock Uncontracted {R}ys
quadrature implementation of up to {G} functions on graphical processing
units. \newblock \abtype{2}{J. Chem. Theory Comput.}
\abtype{3}{6}\abtype{4}{3}, 696-704 \abtype{5}{March 2010}.

\bibitem {https://doi.org/10.1002/jcc.540110809}%
\abtype{0}{J.~D. Augspurger, D.~E. Bernholdt\abphrase{1}\abphrase{0}C.~E.
Dykstra}. \newblock Concise, open-ended implementation of {R}ys polynomial
evaluation of two-electron integrals. \newblock \abtype{2}{J. Comput. Chem.}
\abtype{3}{11}\abtype{4}{8}, 972--977 \abtype{5}{1990}.

\bibitem {MR1186737}\abtype{0}{S.~Belmehdi}. \newblock On semi-classical
linear functionals of class {$s=1$}. {C}lassification and integral
representations. \newblock \abtype{2}{Indag. Math. (N.S.)}
\abtype{3}{3}\abtype{4}{3}, 253--275 \abtype{5}{1992}.

\bibitem {MR1272122}\abtype{0}{ S.~Belmehdi \abphrase{0} A. Ronveaux}.
\newblock Laguerre-Freud's equations for the recurrence coefficients of
semi-classical orthogonal polynomials.
\newblock \abtype{2} {J. Approx. Theory} \abtype{3}{76} \abtype{4}{3}, 351-358 \abtype{5} {1994}.

\bibitem {MR0481884}\abtype{0}{T.~S. Chihara}.
\newblock \abtype{1}{An Introduction to Orthogonal Polynomials}.
\newblock Mathematics and its Applications, Vol. 13. Gordon and Breach Science
Publishers, New York-London-Paris \abtype{5}{1978}.

\bibitem {MR1158212}\abtype{0}{R.~C.~Y. Chin}. \newblock A domain
decomposition method for generating orthogonal polynomials for a {G}aussian
weight on a finite interval. \newblock \abtype{2}{J. Comput. Phys.}
\abtype{3}{99}\abtype{4}{2}, 321--336 \abtype{5}{1992}.

\bibitem {MR4136730}\abtype{0}{D.~Dominici}. \newblock Recurrence coefficients
of {T}oda-type orthogonal polynomials {I}. {A}symptotic analysis.
\newblock\abtype{2}{Bull. Math. Sci.} \abtype{3}{10}\abtype{4}{2}, 2050003, 32
pp. \abtype{5}{2020}.

\bibitem {doi:10.1063/1.432807}%
\abtype{0}{M.~Dupuis, J.~Rys\abphrase{1}\abphrase{0}H.~F. King}.
\newblock Evaluation of molecular integrals over {G}aussian basis functions.
\newblock \abtype{2}{J. Chem. Phys.} \abtype{3}{65}\abtype{4}{1}, 111--116 \abtype{5}{1976}.

\bibitem {doi:10.1063/1.3204437}\abtype{0}{N.~Flocke}. \newblock On the use of
shifted {J}acobi polynomials in accurate evaluation of roots and weights of
{R}ys polynomials. \newblock \abtype{2}{J. Chem. Phys.}
\abtype{3}{131}\abtype{4}{6}, 064107 \abtype{5}{2009}.

\bibitem {MR4331433}\abtype{0}{J.~C. Garc\'{\i}a-Ardila,
F.~Marcell\'{a}n\abphrase{1}\abphrase{0}M.~E. Marriaga}.
\newblock \abtype{1}{Orthogonal polynomials and linear functionals---an
algebraic approach and applications}. \newblock EMS Series of Lectures in
Mathematics. EMS Press, Berlin \abtype{5}{[2021] \copyright 2021}.

\bibitem {MR661060}\abtype{0}{W.~Gautschi}. \newblock A survey of {G}%
auss-{C}hristoffel quadrature formulae.
\newblock \abphrase{7}\abtype{1}{E. {B}. {C}hristoffel ({A}achen/{M}onschau,
1979)}, \abphrase{13} 72--147. Birkh\"{a}user, Basel-Boston, Mass. \abtype{5}{1981}.

\bibitem {MR2061539}\abtype{0}{W.~Gautschi}.
\newblock\abtype{1}{Orthogonal Polynomials: Computation and Approximation}.
\newblock Numerical Mathematics and Scientific Computation. Oxford University
Press, New York \abtype{5}{2004}.

\bibitem {MR2768529}\abtype{0}{M.~Kauers\abphrase{0}P.~Paule}.
\newblock \abtype{1}{The concrete tetrahedron. Symbolic sums, recurrence equations, generating functions, asymptotic
estimates}. \newblock Texts and Monographs in Symbolic Computation. Springer
Wien-NewYork- Vienna \abtype{5}{2011}.

\bibitem {doi:10.1021/acs.jpca.6b10004}\abtype{0}{H.~F. King}.
\newblock Strategies for evaluation of {R}ys roots and weights.
\newblock \abtype{2}{J. Phys. Chem. A} \abtype{3}{120}\abtype{4}{46},
9348--9351 \abtype{5}{2016}. \newblock PMID: 27934243.

\bibitem {MR474710}\abtype{0}{H.~F. King\abphrase{0}M.~Dupuis}.
\newblock Numerical integration using {R}ys polynomials.
\newblock \abtype{2}{J. Comput. Phys.} \abtype{3}{21}\abtype{4}{2}, 144--165 \abtype{5}{1976}.

\bibitem {doi:10.1063/1.3600745}%
\abtype{0}{A.~Komornicki\abphrase{0}H.~F. King}. \newblock A general
formulation for the efficient evaluation of $n$-electron integrals over
products of {G}aussian charge distributions with {G}aussian geminal functions.
\newblock \abtype{2}{J. Chem. Phys.} \abtype{3}{134}\abtype{4}{24}, 244115 \abtype{5}{2011}.

\bibitem {doi:10.1063/1.461610}%
\abtype{0}{R.~Lindh, U.~Ryu\abphrase{1}\abphrase{0}B.~Liu}. \newblock The
reduced multiplication scheme of the {R}ys quadrature and new recurrence
relations for auxiliary function based two-electron integral evaluation.
\newblock \abtype{2}{J. Chem. Phys.} \abtype{3}{95}\abtype{4}{8}, 5889--5897 \abtype{5}{1991}.

\bibitem {MR2474331}\abtype{0}{C.~Lubich}.
\newblock \abtype{1}{From quantum to classical molecular dynamics: reduced
models and numerical analysis}. \newblock Zurich Lectures in Advanced
Mathematics. European Mathematical Society (EMS), Z\"{u}rich \abtype{5}{2008}.

\bibitem {MR1705232}\abtype{0}{A.~P. Magnus}. \newblock Freud's equations for
orthogonal polynomials as discrete {P}ainlev\'{e} equations.
\newblock \abphrase{7}\abtype{1}{Symmetries and integrability of difference
equations ({C}anterbury, 1996)}, \abphrase{8}
255\abphrase{5}\abtype{1}{London Math. Soc. Lecture Note Ser.}, \abphrase{13}
228--243. Cambridge Univ. Press, Cambridge \abtype{5}{1999}.

\bibitem {MR932783}\abtype{0}{P.~Maroni}. \newblock Prol\'{e}gom\`enes \`a
l'\'{e}tude des polyn\^{o}mes orthogonaux semi-classiques.
\newblock \abtype{2}{Ann. Mat. Pura Appl. (4)} \abtype{3}{149}, 165--184 \abtype{5}{1987}.

\bibitem {MR1270222}\abtype{0}{P.~Maroni}. \newblock Une th\'{e}orie
alg\'{e}brique des polyn\^{o}mes orthogonaux. {A}pplication aux polyn\^{o}mes
orthogonaux semi-classiques.
\newblock \abphrase{7}\abtype{1}{Orthogonal polynomials and their applications
({E}rice, 1990)}, \abtype{2}{IMACS Ann. Comput.
Appl. Math. 9} \abphrase{13} 95--130. Baltzer, Basel \abtype{5}{1991}.

\bibitem {MR1246855}\abtype{0}{P.~Maroni}. \newblock Variations around
classical orthogonal polynomials. {C}onnected problems.
\newblock\abphrase{7}\abtype{1}{Proceedings of the {S}eventh {S}panish
{S}ymposium on {O}rthogonal {P}olynomials and {A}pplications ({VII} {SPOA})
({G}ranada, 1991)}, \abtype{2} {J. Comput. Appl. Math.}
\abtype{3}{48}\abtype{4}{1--2}, 133--155 \abtype{5}{1993}.

\bibitem {MR1711161}\abtype{0}{P.~Maroni}. \newblock Semi-classical character
and finite-type relations between polynomial sequences.
\newblock \abtype{2}{Appl. Numer. Math.} \abtype{3}{31}\abtype{4}{3}, 295--330 \abtype{5}{1999}.

\bibitem {MR3937492}\abtype{0}{G.~V. Milovanovi\'{c}}. \newblock An efficient
computation of parameters in the {R}ys quadrature formula.
\newblock \abtype{2}{Bull. Cl. Sci. Math. Nat. Sci. Math.} \abtype{4}{43},
39--64 \abtype{5}{2018}.

\bibitem {MR4412929}%
\abtype{0}{G.~V. Milovanovi\'{c}\abphrase{0}N.~Vasovi\'{c}}.
\newblock Orthogonal polynomials and generalized {G}auss-{R}ys quadrature
formulae. \newblock \abtype{2}{Kuwait J. Sci.} \abtype{3}{49}\abtype{4}{1},
Paper No. 4, 17 \abtype{5}{2022}.

\bibitem {MR2723248}\abtype{0}{F.~W.~J. Olver, D.~W. Lozier, R.~F.
Boisvert\abphrase{1}\abphrase{0}C.~W. Clark}\abphrase{3}.
\newblock\abtype{1}{N{IST} Handbook of Mathematical Functions}. \newblock U.S.
Department of Commerce, National Institute of Standards and Technology,
Washington, DC; Cambridge University Press, Cambridge \abtype{5}{2010}.

\bibitem {doi:10.1063/1.4822929}%
\abtype{0}{W.~H. Press\abphrase{0}S.~A. Teukolsky}. \newblock Orthogonal
polynomials and {G}aussian quadrature with nonclassical weight functions.
\newblock \abtype{2}{Comput. Phys.} \abtype{3}{4}\abtype{4}{4}, 423--426 \abtype{5}{1990}.

\bibitem {doi:10.1063/1.446960}%
\abtype{0}{H.~B. Schlegel, J.~S. Binkley\abphrase{1}\abphrase{0}J.~A. Pople}.
\newblock First and second derivatives of two electron integrals over
{C}artesian {G}aussians using {R}ys polynomials.
\newblock \abtype{2}{J. Chem. Phys.} \abtype{3}{80}\abtype{4}{5}, 1976--1981 \abtype{5}{1984}.

\bibitem {SCHWENKE2014762}\abtype{0}{D.~W. Schwenke}. \newblock On the
computation of high order {R}ys quadrature weights and nodes.
\newblock \abtype{2}{Comput. Phys. Commun.} \abtype{3}{185}\abtype{4}{3},
762--763 \abtype{5}{2014}.

\bibitem {Tenno1993AnEA}\abtype{0}{L.~T.-N. Seiichiro}. \newblock An efficient
algorithm for electron repulsion integrals over contracted {G}aussian-type
functions. \newblock \abtype{2}{Chem. Phys. Lett.} \abtype{3}{211}, 259--264 \abtype{5}{1993}.

\bibitem {MR903848}\abtype{0}{W.~Van~Assche}.
\newblock \abtype{1}{Asymptotics for Orthogonal Polynomials}, \abphrase{8}
1265\abphrase{5}\abtype{1}{Lecture Notes in Mathematics}.
\newblock Springer-Verlag, Berlin \abtype{5}{1987}.

\bibitem {CARSKY1998266}%
\abtype{0}{P.~\v{C}\'{a}rsky\abphrase{0}M.~Pol\'{o}\v{s}ek}.
\newblock Evaluation of molecular integrals in a mixed {G}aussian and
plane-wave basis by {R}ys quadrature. \newblock \abtype{2}{J. Comput. Phys}
\abtype{3}{143}\abtype{4}{1}, 266--277 \abtype{5}{1998}.
\end{thebibliography}
\end{document}